\documentclass[11pt,reqno]{amsart}

\usepackage{amsmath,amsfonts,amsthm,color,amssymb,mathrsfs,mathabx,tikz}
\usepackage{xcolor}
\usepackage{dsfont}
\usepackage{stmaryrd} 
\SetSymbolFont{stmry}{bold}{U}{stmry}{m}{n}
\usepackage[all]{xy}
\usepackage[left=1in, right=1in, top=1.1in, bottom=1.1in]{geometry}
\usepackage{enumitem}
\usepackage{hyperref}
\usepackage{extarrows}

\hypersetup{
colorlinks   = true,
citecolor    = blue,
linkcolor=blue
}
 \usepackage{changepage}
\allowdisplaybreaks
\numberwithin{equation}{section}

\newtheorem{theorem}{Theorem}[section]
\newtheorem{lemma}[theorem]{Lemma}

\newtheorem{proposition}[theorem]{Proposition}

\newtheorem{definition}[theorem]{Definition}

\newtheorem{remark}[theorem]{Remark}
\newtheorem{notations}[theorem]{Notations}
\newtheorem{assumption}{Assumption}
\newtheorem{assumptionalt}{Assumption}[assumption]
\newenvironment{assumptionp}[1]{
\renewcommand\theassumptionalt{#1}
\assumptionalt
}{\endassumptionalt}

\def\bZ{{\mathbb Z}}
\def\I{\mathbf{I}}
\def\bP{{\mathbb P}}
\def\rme{{\rme}}\def\al{{\alpha}}

\def\la{\left\langle}
\def\ra{\right\rangle}
\def\e{\varepsilon}

\def\bR{\mathbb R}
\def\bN{\mathbb N}
\def\bE{\mathbb E}\def\E{\mathbb E}

\def\bI{\mathbb I}

\def\cH{\mathcal H}

\def\wc{\overset{(d)}{\longrightarrow}}

\def\dd{{\mathrm d}}
\def\zf{Z_{N,\beta}^{\omega,h}}
\def\zc{Z_{N,\beta}^{\omega,h,\text{c}}}
\def\nzf{\tilde{Z}_{N,\beta}^{\omega,h}}

\newcounter{bean}
\newcommand{\benuma}{\setlength{\labelwidth}{.25in}

\begin{list}
{(\alph{bean})}{\usecounter{bean}}}
\newcommand{\eenuma}{\end{list}}

\begin{document}

\title[Disordered pinning model]{On a pinning model in correlated Gaussian random environments}

\author[Z. Li]{Zi'an Li}\address{Research Center for Mathematics and Interdisciplinary Sciences, Shandong University, Qingdao 266237, China}
\email{202117106@mail.sdu.edu.cn}

\author[J. Song]{Jian Song}\address{Research Center for Mathematics and Interdisciplinary Sciences, Shandong University, Qingdao 266237, China}
\email{txjsong@sdu.edu.cn}

\author[R. Wei]{Ran Wei}\address{Department of Financial and Actuarial Mathematics, Xi'an Jiaotong-Liverpool University, 111 Ren'ai Road, Suzhou 215123, China}
\email{ran.wei@xjtlu.edu.cn}

\author[H. Zhang]{Hang Zhang}\address{Research Center for Mathematics and Interdisciplinary Sciences, Shandong University, Qingdao 266237, China}
\email{202217118@mail.sdu.edu.cn}

\begin{abstract}
We consider a pinning model in correlated Gaussian random environments. For the model that is disorder relevant, we study its intermediate disorder regime and show that the rescaled partition functions converge to a non-trivial continuum limit in the Skorohod setting and in the Stratonovich setting, respectively. Our results partially confirm the Weinrib--Halperin prediction for disorder relevance/irrelevance.
\end{abstract}
\maketitle
\tableofcontents

\section{Introduction}

In this paper, we study a pinning model in Gaussian random environments (also called the \textit{disorder}), of which the correlation has a power-law decay.  We shall show that, under proper scalings, the rescaled partition functions converge weakly to continuum limits which can be represented by chaos expansions. Our results complement those of \cite{Ber13} and are consistent with the prediction of \cite{WH83} on the disorder relevance/irrelevance for physical systems with correlated disorders.

\subsection{The disordered pinning model}\label{sec1.1}

The disordered pinning model is a type of random polymer models, which is defined via a Gibbs transform of the distribution of a renewal process. Pinning models are well suited models for studying phase transitions and critical phenomena and also have received widespread attention in physics, chemistry and biology. We refer to the monographs of \cite{Gia07} and \cite{Gia11} for more details.

We now introduce the setting of our model. Let $\tau=\{\tau_{i}\text{,}~i\in\bN_{0}\}$ be a persistent  renewal process (i.e. $\bP_{\tau}(\tau_{1}<\infty)=1$) on $\bN$ with $\tau_{0}=0$ and the inter-arrival law
\begin{equation}\label{eq:pin_dist}
\bP_{\tau}(\tau_{1}=n)=\frac{L(n)}{n^{1+\al}},\quad\text{for}~n\geq1, 
\end{equation}where $\alpha>0$ and $L(\cdot): \bR^+\to\bR^+$ is a slowly varying function, i.e., $\lim_{t\to\infty}L(at)/L(t)=1$ for any $a>0$ (see \cite{BGT89} for more details about slowly varying functions).  The expectation with respect to $\bP_\tau$ is denoted by $\bE_\tau$. Note that $\bE_{\tau}[\tau_1]<\infty$ for $\alpha>1$ and $\bE_{\tau}[\tau_1]=\infty$ for $\alpha<1$, and by \cite[Proposition 1.5.9a]{BGT89}, for $\alpha=1$,
\begin{equation}\label{e:lN}
\bE_{\tau}[\tau_1\mathbf{1}_{\{\tau_1\leq N\}}]=l(N):=\sum_{n=1}^N\frac{L(n)}{n}
\end{equation}is also slowly varying (whether it is finite or not depends on $L(\cdot)$). Hence, the renewal process $\tau$ can be classified by the following conditions:
\vspace{-0.1cm}
\begin{equation}
\text{(1) $0<\al<1$;\quad(2) $\E_\tau[\tau_1]<\infty$;\quad(3) $\alpha=1$ and  $\E_\tau[\tau_1]=\infty$.}
\end{equation}

\vspace{-0.15cm}
\noindent Noting that our renewal process $\tau$ is \emph{aperiodic} since $L(n)>0$ for all $n$, we have the renewal theorem (see \cite{Don97} and \cite[Theorem 8.7.5]{BGT89}):
\vspace{-0.1cm}
\begin{eqnarray}\label{local}
\bP_\tau(n\in\tau)\sim
\begin{cases}
\dfrac{C_\alpha}{L(n)n^{1-\alpha}}, & \text{if }0<\alpha<1~\Big(C_{\alpha}=\frac{\alpha\sin(\pi\alpha)}{\pi}\Big),\\[10pt]
\dfrac{1}{\bE[\tau_{1}]}, & \text{if }\bE[\tau_{1}]<\infty,\\[10pt]
\dfrac{1}{l(n)}, & \text{if } \alpha=1 \text{ and } \E[\tau_1]=\infty,
\label{renewl theorem}
\end{cases}
\quad \qquad\text{as}~n\to\infty.
\end{eqnarray}

\vspace{-0.15cm}
Let $\omega=\{\omega_{n}\text{,}~n\in\bN\}$ be a family of Gaussian random variables independent of $\tau$, representing the random environment. The probability and expectation for $\omega$ are denoted by $\bP_\omega$ and $\bE_\omega$, respectively. We assume that
\vspace{-0.1cm}
\begin{equation}\label{eq:omega_dist}
\bE_\omega[\omega_n]=0\quad\text{and}\quad\bE_\omega[\omega_n\omega_{n'}]=\gamma(n-n'),\quad\text{for } n,n'\in \bN,
\end{equation}

\vspace{-0.15cm}
\noindent where the covariance function $\gamma(\cdot)$ satisfies
\vspace{-0.1cm}
\begin{equation}\label{con-gamma}
\gamma(n)\lesssim |n|^{2H-2}\wedge1, \text{ for } n\in\bZ~\text{(see Notations \ref{notations} for the meaning of ``$\lesssim$'').}
\end{equation}

\vspace{-0.15cm}
\noindent Here, $H\in(1/2,1)$ is a fixed constant (called the \textit{Hurst parameter}) throughout this paper. We further assume that for all $t\neq 0$,
\vspace{-0.1cm}
\begin{equation}\label{con-gamma'}
\lim\limits_{N\rightarrow\infty}N^{2-2H}\gamma(\left[Nt\right])=|t|^{2H-2}.
\end{equation}

\vspace{-0.15cm}
We shall use $\bP=\bP_\tau\otimes\bP_\omega$ and $\bE=\bE_\tau\otimes\bE_\omega$ on the product probability space of $(\tau,\omega)$. We may also use $\bP$ and $\bE$ to denote the probability and the expectation respectively of some random variables only involving $\tau$ or $\omega$ when there is no ambiguity.

Denote $\Omega_{N}:=\{\frac{1}{N}\text{,}\frac{2}{N}\text{,}\cdots\text{,}\frac{N-1}{N}\text{,}1\}$ and $\delta_x^{(N)}=\delta_{x}:=\mathbf{1}_{\{Nx\in\tau\}}$.  The disordered pinning model with \emph{free} boundary is defined via Gibbs transform by
\vspace{-0.1cm}
\begin{equation}\label{def:Gibbs}
\frac{\dd\bP_{N,\beta}^{\omega,h}}{\dd\bP_\tau}(\tau):=\frac{1}{\zf}\exp\Big\{\sum_{x\in\Omega_{N}}(\beta \omega_{Nx}+h)\delta_{x}\Big\},
\end{equation}

\vspace{-0.15cm}
\noindent where
\vspace{-0.1cm}
\begin{equation}\label{e:partition}
\zf=\bE_{\tau}\Big[\exp\Big\{\sum_{x\in\Omega_{N}}(\beta \omega_{Nx}+h)\delta_{x}\Big\}\Big]
\end{equation}

\vspace{-0.15cm}
\noindent is the \emph{partition function},   $\beta>0$ is the inverse temperature and $h\in\bR$ is an external field. One can also consider the model with conditioned boundary, of which the partition function is given by
\vspace{-0.1cm}
\begin{equation}\label{def:cond}
\zc=\bE_{\tau}\Big[\exp\Big\{\sum_{x\in\Omega_{N}}(\beta \omega_{Nx}+h)\delta_{x}\Big\}\Big|N\in\tau\Big].
\end{equation}

\subsection{Disorder relevance/irrelevance and the Weinrib--Halperin prediction}\label{sec12}

When the disorder is i.i.d., the \emph{Harris criterion} \cite{Harris74} indicates that disorder is \emph{irrelevant} if $\al<\frac{1}{2}$ and is \emph{relevant} if $\al>\frac{1}{2}$ (with $\alpha=\frac12$ being the marginal case). Here,  ``disorder is relevant''  means that for any $\beta>0$, the behavior of $\tau$ under $\bP_{N,\beta}^{\omega,h}$ differs drastically from that under  $\bP_\tau$. To be specific, some crucial quantities, such as the \emph{correlation length} and the \emph{critical point shift} will be changed by introducing disorder in the Gibbs transform \eqref{def:Gibbs}. When disorder is irrelevant, the behavior of $\tau$ under $\bP_{N,\beta}^{\omega,h}$ is comparable to that under $\bP_\tau$, at least for small enough $\beta>0$. Readers may refer to \cite{CSZ16,CSZ16+,BL18,CTT17} for more detailed discussion about disorder relevance and irrelevance.

When the disorder is correlated and the correlation is of the form $n^{-p}$ ($p > 0$ and $n$ is the distance between two points), \cite{WH83} predicted that the disorder should be
relevant if $\al>\min(p, 1)/2$ and irrelevant if $\al<\min(p, 1)/2$. \cite{Ber13} confirmed the Weinrib--Halperin prediction for the case $p>1$ (note that in this case the Harris criterion still applies), and showed that the case $p<1$ is different in the sense that classical methods (see, e.g., \cite{Gia11}) in the study of disorder relevance/irrelevance by comparing critical points/exponents no longer apply: when $p<1$, the disordered system does not possess a localization/delocalization phase transition, since the annealed free energy is always infinite and the quenched free energy is always strictly positive for all $h$ as long as $\beta>0$, and thus the system is always localized (see Theorem 2.5 and Section 2.3 in \cite{Ber13}). This phenomenon is due to the fact, as pointed out by \cite{Ber14},  that when $p<1$, there are large and frequent regions in the environments with arbitrarily high attractions (called the ``infinite disorder'' in \cite{Ber14}). These results indicate that different methods are needed to study  disorder relevance/irrelevance for the case $p<1$.

On the other hand, \cite{CSZ16} proposed a new perspective on disorder relevance/irrelevance: it is possible that if the strength of disorder is tuned down to zero at a proper rate (i.e., $\beta_N\to0$) as the system size tends to infinity, then the disorder will persist in the resulting continuum model. This phenomenon has been shown for the i.i.d.\ disorder by studying weak limits of rescaled partition functions for the directed polymer model in \cite{AKQ14} and for pinning models, directed polymer models and the Ising model in a unified way in \cite{CSZ16}. For related results on directed polymers,  we refer to \cite{rang20,  cg23} for the case when the disorder is correlated in space but independent in time, to \cite{rsw24} for the case when the disorder is correlated in time but independent in space, and to \cite{SSSX21} for the case when the disorder is space-time correlated. In those models,  the Weinrib--Halperin prediction is not clear (the decay rates in space and in time of the correlation of the disorder in \cite{SSSX21}  do not agree).

Inspired by the above-mentioned works, we study the weak convergence of the rescaled partition functions. In our setting, noting that $p=2-2H\in(0,1)$ with $H\in(\frac{1}{2},1)$, the Weinrib--Halperin prediction suggests that disorder  is relevant  if $\alpha>1-H$, which coincides with the condition for the existence of the weak limit for each chaos (see Proposition~\ref{prop:weak-con-Sk}). When $H\in(\frac12,1)$, under the stronger condition $\alpha\geq\frac12$,  we establish in Theorem~\ref{thm1} that the rescaled Wick-ordered partition functions converge weakly to an $L^2$ random variable. We further conjecture that for  $H\in(\frac12,1)$, these rescaled Wick-ordered partition functions converge weakly to an $L^1$ random variable under the Weinrib--Halperin  condition $\alpha>1-H$ (see Remark~\ref{rem:highlight} and Section~\ref{sec:discussion} for further details).

Finally, we make a few more remarks on disorder relevance/irrelevance for pinning models in correlated random environment. \cite{Poi12} showed that the Harris criterion still applies to the pinning model in random environments with \emph{finite range} correlation. When disorders have \emph{long range} correlation, \cite{BP15} confirmed the Weinrib--Halperin prediction for the case when the correlations are \emph{summable}, which extended the result in \cite{Ber13}, and \cite{Poi13+} obtained sharp asymptotics of the annealed free energy for doubly summable or exponentially decaying correlations. We also mention that \cite{CCP19} studied a disordered pinning model, where the random environment is given by another independent renewal process. In that case, the random environment is correlated and it is different from the more typical setting where the random environment is a family of random variables. They  found that in their model, the chaos expansion heuristics and the Weinrib--Halperin prediction provide different phase diagram.


\subsection{Main results}
 Rather than studying the Gibbs measure directly, which is highly involved, we focus on its partition function, which already encapsulates sufficient information about the model. Noting that
\vspace{-0.1cm}
\begin{equation*}
\bE_\omega\Big[\exp\Big\{\sum\limits_{x\in\Omega_N}\beta\omega_{Nx}\delta_x\Big\}\Big]=\exp\Big\{\frac12\mathrm{Var}_{\omega}\Big(\sum\limits_{x\in\Omega_N}\beta\omega_{Nx}\delta_x\Big)\Big\},
\end{equation*}

\vspace{-0.15cm}
\noindent we first consider the following \emph{Wick-ordered} partition function
\vspace{-0.1cm}
\begin{equation}\label{e:partition'}
\nzf:=
\bE_{\tau}\Big[\exp\Big\{ \sum_{x\in\Omega_{N}}(\beta \omega_{Nx}+h)\delta_{x} -\frac12 \mathrm{Var}_{\omega}\Big(\sum_{x\in\Omega_{N}}\beta \omega_{Nx}\delta_{x}\Big)\Big\}\Big],
\end{equation}

\vspace{-0.15cm}
\noindent where the term $\frac12\mathrm{Var}_\omega(\sum_{x\in\Omega_N}\beta\omega_{Nx}\delta_x)=\frac{1}{2}\beta^2\sum_{n,m=1}^N \gamma(n-m) \mathbf 1_{\{n,m\in \tau\}}$ is a direct normalization of the partition function (we call it \textit{Wick normalization}), which depends on the realization of $\tau$. We adopt this normalization because disorder is correlated, and thus a polynomial chaos expansion as has been done in \cite{AKQ14, CSZ16} cannot be employed for our pinning model. The normalization in the Wick-ordered partition function \eqref{e:partition'} results in a Wick exponential, which allows us to perform a Wick expansion (see Section \ref{sec:skor} and also the discussion in \cite{LMS24}). We point out that the Wick-ordered partition function~\eqref{e:partition'} reduces to \cite[equation (3.3)]{CSZ16} when the disorders are i.i.d., noting that $\frac12 \mathrm{Var}_{\omega}(\sum_{x\in\Omega_{N}}\beta \omega_{Nx}\delta_{x})=\frac{1}{2}\beta^2\sum_{n=1}^N \mathbf 1_{\{n\in\tau\}}$. After addressing the Wick-ordered partition function, we turn to the original partition function $Z_{N,\beta}^{\omega,h}$ given in \eqref{e:partition}, for which we carry out a Taylor expansion. The finite-order chaos in the Wick and Taylor expansions is connected by the Hu--Meyer formula (see Eq.~\eqref{H-M} in Section~\ref{sec:pre}).

Let $\hat{\beta }>0$ and $\hat{h }\in\bR$ be fixed constants throughout this paper.  Consider the following scalings:
\vspace{-0.1cm}
\begin{eqnarray}\label{scale:beta}
\beta_{N}=
\begin{cases}
\hat{\beta }\dfrac{L(N)}{N^{\alpha+H-1}}, & \text{if  }  0<\alpha<1,\\[10pt]
\hat{\beta }\dfrac{1}{N^{H}}, &\text{if  }\bE[\tau_{1}]<\infty,\\[10pt]
\hat\beta \dfrac{l(N)}{N^H}, &\text{if } \alpha=1 \text{ and } \E[\tau_1]=\infty,
\end{cases}
h_{N}=
\begin{cases}
\hat{h}\dfrac{L(N)}{N^{\alpha}},& \text{if } 0<\alpha<1,\\[10pt]
\hat{h}\dfrac{1}{N},&\text{if } \bE[\tau_{1}]<\infty,\\[10pt]
\hat h \dfrac{l(N)}{N}, &\text{if } \alpha=1 \text{ and } \E[\tau_1]=\infty,
\end{cases}
\end{eqnarray}

\vspace{-0.15cm}
\noindent where we recall that  $l(n)$  given in \eqref{e:lN} is a slowly varying function.

Let $ \phi_{k}(t_1,\dots, t_k)$ be a \emph{symmetric} function. For any  given $(t_1,\cdots,t_k)\in(0,1)^k$ with $t_i\neq t_j$ for any $i\neq j$ (if there exist $i\neq j$, s.t.\ $t_i= t_j$, we set $\phi_k=0$), there exists a unique permutation $\sigma$ on $\llbracket k\rrbracket$ such that $0<t_{\sigma(1)}<t_{\sigma(2)}<\cdots<t_{\sigma(k)}<1$. We define $\phi_k$ as follows:
\vspace{-0.1cm}
\begin{equation}\label{e:phi}
\begin{aligned}
\phi_{k}(t_{1},\dots,t_{k})
:=
\begin{cases}
\displaystyle\frac{C_{\alpha}^{k}}{ t_{\sigma(1)}^{1-\alpha}(t_{\sigma(2)}-t_{\sigma(1)})^{1-\alpha}\cdots(t_{\sigma(k)}-t_{\sigma(k-1)})^{1-\alpha}} , &\text{ if }   1-H<\alpha<1,\\[10pt]
\displaystyle\Big(\frac{1}{\E[\tau_1]}\Big)^{k},& \text{ if } \E[\tau_1]<\infty,\\[10pt]
\displaystyle 1, & \text{ if }  \alpha=1 \text{ and } \E[\tau_1]=\infty,
\end{cases}
\end{aligned}
\end{equation}

\vspace{-0.15cm}
\noindent where $C_{\alpha}$ is the same as that in \eqref{local}.

Now we present our main results for the Wick-ordered partition function $\nzf$  given in \eqref{e:partition'} and for the original partition function $\zf$ given in \eqref{e:partition}, respectively. The weak limits $\tilde Z_{\hat\beta, \hat h}$ in \eqref{e:Z1} and $Z_{\hat\beta, \hat h}$ in \eqref{e:Z2} involve multiple stochastic integrals in the sense of \emph{Skorohod} and \emph{Stratonovich} (see Section~\ref{sec:pre} for their definitions), respectively. Roughly speaking,  the multiple \emph{Skorohod} integrals in \eqref{e:Z1} (resp. the multiple \emph{Stratonovich} integrals in \eqref{e:Z2}) can be defined as the limit of Riemann sums, where the products are taken as Wick products (resp. ordinary products). Although multiple Skorohod and Stratonovich integrals are related through the Hu–Meyer formula \eqref{H-M}, it is not straightforward to establish a direct connection between $\tilde Z_{\hat \beta, \hat h}$ and $Z_{\hat \beta, \hat h}$.

\begin{theorem}\label{thm1}
Let the renewal process $\tau$ and the disorder $\omega$ be given as in Section \ref{sec1.1}. Let H be the parameter in \eqref{con-gamma}-\eqref{con-gamma'} with $H\in(\frac12, 1)$. Recall the Wick-ordered partition function $\tilde{Z}^{\omega,h}_{N,\beta}$ given in \eqref{e:partition'} and the scalings \eqref{scale:beta}. We have the following three cases for different $\alpha>0$ in \eqref{eq:pin_dist}.

\textbf{Case (1):} If $\alpha>\frac12$, then for any $\hat{\beta}>0$ and $\hat{h}\in\bR$, we have that
\vspace{-0.2cm}
\[\tilde{Z}^{\omega,h_N}_{N,\beta_N}\wc\tilde{Z}_{\hat{\beta },\hat{h }} ~,\text{ as } N\to\infty.\]

\vspace{-0.15cm}
\noindent Here $\tilde{Z}_{\hat{\beta },\hat{h }}$ is given by the  $L^2$-convergent series:
\vspace{-0.2cm}
\begin{equation}\label{e:Z1}
\tilde{Z}_{\hat{\beta },\hat{h }}:=1+\sum_{k=1}^{\infty}\frac{1}{k!}\int_{[0,1]^{k}}\phi_{k}(t_{1},\cdots,t_{k})\prod\limits_{i=1}^k\diamond\left(\hat{\beta }W(\dd t_{i})+\hat{h }\dd t_{i}\right),
\end{equation}

\vspace{-0.1cm}
\noindent where $\phi_k$ is given by \eqref{e:phi}, $\prod\diamond$ is the \emph{Wick} product, $\dot W$ is Gaussian noise with $\E[\dot W(t)\dot W(s)]=|t-s|^{2H-2}$ and  the stochastic integral is of \emph{Skorohod} type.

\textbf{Case (2):} If $\alpha=\frac12$ and we further assume that the slowly varying function $L(\cdot)$ in \eqref{eq:pin_dist} satisfies $0<\inf_{n\in\bN}L(n)\leq\sup_{n\in\bN}L(n)<+\infty$, then there exist some $\hat{\beta}_c,\hat{h}_c\in(0,+\infty)$, such that for $\hat{\beta}\in(0,\hat{\beta}_c)$ and $|\hat{h}|<\hat{h}_c$, we have the same weak convergence as in \textbf{Case (1)}, that is,
\vspace{-0.15cm}
\begin{equation*}
\tilde{Z}^{\omega,h_N}_{N,\beta_N}\wc\tilde{Z}_{\hat{\beta },\hat{h }} ~,\text{ as } N\to\infty,
\end{equation*}

\vspace{-0.15cm}
\noindent where $\tilde{Z}_{\hat{\beta},\hat{h}}$ is given by the $L^2$-convergent series \eqref{e:Z1}.

\textbf{Case (3):} If $\alpha<\frac12$, then for any $\hat{\beta}>0$ and $\hat{h}\in\bR$, we have 
\vspace{-0.15cm}$
\lim\limits_{N\to\infty}\bE\Big[\Big(\tilde{Z}^{\omega,h_N}_{N,\beta_N}\Big)^2\Big]=+\infty.$
\end{theorem}

\begin{remark}

Note that $\alpha = \tfrac{1}{2}$ is the threshold between the \emph{$L^2$ regime} and the \emph{non-$L^2$ regime}, which will be further discussed in Remark~\ref{remark:alpha}. The case $\alpha < \tfrac{1}{2}$ will be addressed in Remark~\ref{rem:highlight} below.
\end{remark}

\begin{theorem}\label{thm2}
Let the renewal process $\tau$ and the disorder $\omega$ be given as in Section \ref{sec1.1}. Let H and $\alpha$ be parameters satisfying $H\in(\frac12, 1)$ and $\alpha+H>\frac32$. Recalling the partition function $Z^{\omega,h}_{N,\beta}$ given in \eqref{e:partition} and the scalings \eqref{scale:beta}, we have that for any $\hat{\beta}>0$ and $\hat{h}\in\bR$,
\vspace{-0.1cm}
\[Z^{\omega,h_N}_{N,\beta_N}\wc Z_{\hat{\beta },\hat{h }} ~,~as~N\to\infty.\]

\vspace{-0.15cm}
\noindent Here $Z_{\hat{\beta },\hat{h }}$ is given by the $L^2$-convergent series:
\vspace{-0.1cm}
\begin{equation}\label{e:Z2}
Z_{\hat{\beta },\hat{h }}:=\sum_{k=0}^{\infty}\frac{1}{k!}\int_{[0,1]^{k}}\phi_{k}(t_{1},\cdots,t_{k})\prod_{i=1}^{k}\left(\hat{\beta }W(\dd t_{i})+\hat{h }\dd t_{i}\right),
\end{equation}

\vspace{-0.15cm}
\noindent where $\phi_k$ is given by \eqref{e:phi}, $\dot W$ is Gaussian noise with $\E[\dot W(t)\dot W(s)]=|t-s|^{2H-2}$, and  the stochastic integral is of \emph{Stratonovich} type.
\end{theorem}

\begin{remark}\label{rmk:A}
Assuming the conditions $H\in(\frac12, 1)$ and $\alpha+H>\frac32$ as in Theorem \ref{thm2}, the rescaled original partition function $Z^{\omega,h_N}_{N,\beta_N}$ is actually uniformly bounded in  $L^1$ (see  Section \ref{Lp_bound}), and thus no normalization is needed in this situation. 

In contrast, under the weaker conditions $H \in \bigl(\tfrac12,1\bigr)$ and $\alpha \geq \frac12$ assumed in Cases (1) and (2) in Theorem~\ref{thm1}, Wick normalization is employed to obtain weak convergence to a nontrivial random variable. In our setting with colored Gaussian noise, Wick normalization provides more convenience for computations than the “standard” normalization. Furthermore, since the polymer measure acts as a random measure of the underlying renewal process, it is natural to adopt Wick normalization so that the normalized partition function has a mean equal to the \emph{homogeneous} partition function (i.e., $Z_{N,\beta=0}^{\omega,h}$), which is $1$ when $h=0$. This viewpoint has also been adopted very recently in \cite{qrv} for the parabolic Anderson model on $\mathbb R^2$, where the classical Skorohod integral was extended to the $L^1$ setting and the global existence of an $L^1$ Skorohod solution of the stochastic heat equation with multiplicative planar white noise was proved.

For the directed polymer model in time-correlated random environment, the continuum Skorohod (resp. Stratonovich) partition function corresponds to the Skorohod (resp. Stratonovich) solution of the stochastic heat equation driven by time-colored Gaussian noise (see \cite{rsw24}).  However, the literature on SPDEs related to pinning models in random environments is very limited. A stochastic heat equation with multiplicative white noise, corresponding to the marginally relevant pinning model with $\alpha=\tfrac12$ in an independent random environment, was recently studied  in \cite{WY24}. In our forthcoming work \cite{LSWZ}, we investigate a class of fractional stochastic heat equations in the sense of Skorohod (resp. Stratonovich) associated with the Skorohod (resp. Stratonovich) continuum partition functions of random pinning models.
\end{remark}

\begin{remark}
The results in Theorems \ref{thm1} and \ref{thm2} also hold for the conditioned partition functions
\vspace{-0.1cm}
$$
\tilde{Z}_N^{\text{c}}=\tilde{Z}_{N,\beta_N}^{\omega,h_N,\text{c}}:=
\bE_{\tau}\Big[\exp\Big\{ \sum_{x\in\Omega_{N}}(\beta \omega_{Nx}+h)\delta_{x} -\frac12 \mathrm{Var}\Big[\sum_{x\in\Omega_{N}}(\beta \omega_{Nx})\delta_{x}\Big]\Big\}\Big| N\in\tau\Big],~\text{and}$$

\vspace{-0.3cm}
$$
{Z}_N^{\text{c}}={Z}_{N,\beta_N}^{\omega,h_N,\text{c}}:=
\bE_{\tau}\Big[\exp\Big\{ \sum_{x\in\Omega_{N}}(\beta \omega_{Nx}+h)\delta_{x} \Big\}\Big| N\in\tau\Big],$$

\vspace{-0.15cm}
\noindent with the integrand $\phi_k$ in \eqref{e:Z1} and \eqref{e:Z2} being replaced by a \emph{symmetric} function $\phi_k^c$  with the same symmetrization as for $\phi_k$, which for $0< t_1<\cdots<t_{k}<1$ is defined by
\vspace{-0.1cm}
\begin{equation*}
\begin{aligned}
\phi_{k}^c(t_{1},\dots,t_{k})
:=
\begin{cases}
\displaystyle\frac{C_{\alpha}^{k}}{t_{1}^{1-\alpha}(t_{2}-t_{1})^{1-\alpha}\cdots(t_{k}-t_{k-1})^{1-\alpha}(1-t_k)^{1-\alpha}} , & \text{ if } 1-H<\alpha<1,\\[10pt]
\displaystyle\Big(\frac{1}{\E[\tau_1]}\Big)^{k},& \text{ if } \E[\tau_1]<\infty,\\[10pt]
\displaystyle 1, & \text{ if }  \alpha=1 \text{ and } \E[\tau_1]=\infty.
\end{cases}
\end{aligned}
\end{equation*}
\end{remark}

\vspace{-0.3cm}
\begin{remark}\label{rem:highlight}
A natural condition for the rescaled Wick-ordered partition  function to converge weakly is $\alpha+H-1>0$, where $\alpha+H-1$ is the scaling exponent for $\beta$ (see \eqref{scale:beta}). Note that for directed polymers, under the condition that the corresponding scaling exponent is positive, the continuum partition  converges in $L^2$ as a Wiener chaos (see \cite{CSZ16, cg23, rsw24}). However, for our pinning model, assuming $\alpha+H-1>0$ we can only prove that each chaos is a well-defined Skorohod integral which is $L^2$-integrable (see Lemma~\ref{lem:norm_degenerate_trace}), and we need a stronger condition $H\in(\frac12,1),\alpha\geq\frac12$ to guarantee the $L^2$-convergence as stated in Theorem~\ref{thm1}.  On the other hand,  when $H\in(\frac12,1)$, we show in Proposition \ref{prop:UI} that, assuming $\alpha>2-2H$,   $\big\{\tilde Z_{N,\beta_N}^{\omega, h_N}\big\}_N$ is uniformly integrable. Note that in this situation, when $H\in(\frac34,1)$,  $\alpha$ may be strictly less than $\frac12$, and thus $\big\{\tilde Z_{N,\beta_N}^{\omega, h_N}\big\}_N$ is not $L^2$-bounded.

This is a new phenomenon that does not appear in the models studied in  \cite{AKQ14,CSZ16,rang20,rsw24,cg23}: in those models, the rescaled partition functions are always uniformly bounded in $L^2$ if the models are in the disorder relevant regime, no matter whether the random environments are independent or correlated. This hints that the pinning model in correlated environment indeed carries some new features.

We note that a phase transition from the $L^2$ regime to an $L^p$ regime with $1<p<2$ has been studied in the directed polymer model with fixed inverse temperature $\beta$. In particular, \cite{J25} showed that, while the polymer measure exhibits the same macroscopic behavior in both the $L^2$ and $L^p$ regimes, its microscopic behavior differs, as evidenced by distinct exponents in the corresponding  local limit theorems. By analogy, we expect that this new phenomenon of the gap between the $L^2$ regime and  the $L^1$ regime in our model should  relate to some  phase transition.

Motivated by the Weinrib--Halperin prediction and the above discussion, we conjecture that under the condition $\alpha+H-1>0$ with $\alpha\leq\frac12$, the rescaled Wick-ordered partition  function $\tilde Z_{N,\beta_N}^{\omega, h_N}$ converges weakly to some $L^1$-integrable random variable  for any $\hat{\beta}>0$.  We will discuss the conjecture in more detail in Section~\ref{sec:discussion}.
\end{remark}

To close this section, let us highlight our contributions in this paper. Firstly, we obtain some finer property of a disordered pinning model with environment correlation of the form $n^{-p}$, where $p<1$, which complements the result of \cite{Ber13}. To be specific, we find scaling limits of rescaled (Wick-ordered) partition functions in the disorder relevant regime. In particular, the Weinrib--Halperin prediction is partially confirmed when it differs from the Harris criterion. We also show that  the rescaled  (Wick-ordered) partition functions are not $L^2$-bounded in the entire disorder relevance regime (see Remark \ref{remark:alpha}), which is a new feature compared to other disordered models. Secondly, we develop a decoupling approach to prove uniform integrability (see Section~\ref{sec:discussion}),  and provide some sharp upper bounds for the norm of traces (see Lemma \ref{lem:Tr^j}) and for the second moments of multiple Stratonovich integrals  (see Lemma  \ref{Tr}). These  auxiliary results may be of independent interest.

\begin{notations}\label{notations}
Let $\bN $ denote the set of natural numbers without 0, i.e., $\bN:=\{1,2,\cdots\}$ and $\bN_{0}:=\{0,1,2,\cdots\} $; for $N\in\bN$, $\llbracket N \rrbracket:=\{1,2,\cdots,N\}$;  For $a\in\bR$, $\left[a\right]$means the greatest integer that is not greater than $a$; $\|\cdot\|$ is used for the Euclidean norm; let\ $\boldsymbol{k}:=(k_{1},k_{2},\cdots,k_{d}), \boldsymbol{x}:=(x_{1},x_{2},\cdots,x_{d})$ etc.\ stand for vectors in $\bZ^{d}$ or $\bR^{d}$ depending on the context; we use C to denote a generic positive constant that may vary in different lines; we say $f(x)\lesssim g(x)$ if $f(x)\leq Cg(x)$ for all $x$; we write $f(x)\sim g(x) (as\  x\rightarrow\infty) $ if $\lim_{x\rightarrow\infty}f(x)/g(x)=1$; we use $\overset{(d)}{\longrightarrow}$ to denote the convergence in distribution (also called the \emph{weak convergence}) for random variables or random vectors. For a random variable X, $\|X\|_{L^{p}}:=(\mathrm{E}\left[|X|^{p}\right])^{1/p}$ for $p\geq 1$. We use $\mathrm{I}_A$ to denote the indicator function, i.e., $\mathrm{I}_{A}(\boldsymbol{t})=1$ if $\boldsymbol{t}\in A$ and $\mathrm{I}_{A}(\boldsymbol{t})=0$ if $\boldsymbol{t}\notin A$.
\end{notations}

The rest of the paper is organized as follows. In Section \ref{sec:pre}, we introduce some preliminaries. The proofs of Theorem \ref{thm1} and Theorem \ref{thm2} will be presented in Section \ref{sec:skor} and Section \ref{sec:Str}, respectively. We will only deal with the case $0<\alpha<1$ in the proofs, noting that the other case $\alpha\geq1$ can be treated in the same way. Some technical estimates are put in Appendix \ref{app:A}.

\section{Preliminaries on Gaussian spaces and U-statistics}\label{sec:pre}

\subsection{Gaussian spaces}\label{sec:Gaussian}

In this subsection, we collect some preliminaries on the analysis of Gaussian space. We refer to \cite{Nua06, Hu16} for more details. 

On a probability space $(\Omega,\mathcal{F},\bP)$ satisfying the usual conditions, let $ \dot{W}=\{\dot{W}(t):t\in[0,1]\}$ be a real-valued centered Gaussian noise with covariance
\vspace{-0.1cm}
\begin{equation}\label{def:noise}
\E[\dot W(t)\dot W(s)]=K(t-s)=|t-s|^{2H-2},
\end{equation}

\vspace{-0.15cm}
\noindent where $H\in(\frac12,1)$ is the Hurst parameter appearing in \eqref{con-gamma} and \eqref{con-gamma'}. Thus, $\dot{W}$ is the continuum counterpart of the disorder $\omega$ introduced in Section \ref{sec1.1}.

The Hilbert space $\mathcal{H}$ associated with $\dot{W}$ is the completion of smooth functions with compact support under the inner product:
\vspace{-0.1cm}
\begin{equation}\label{def:H_inner}
\begin{aligned}
\left<f,g\right>_{\mathcal{H}}&=\int_{0}^{1}\int_{0}^{1}f(s)K(s-t)g(t)\dd s\dd t=\int_{0}^{1}\int_{0}^{1}f(s)g(t)|s-t|^{2H-2}\dd s\dd t.
\end{aligned}
\end{equation}

\vspace{-0.15cm}
\noindent Noting that $\mathcal{H}$ contains distributions (see \cite{pt00}),  it is convenient to consider  $|\cH|$ which is a linear subspace of measurable functions $f$ on $[0,1]$ satisfying
\vspace{-0.1cm}
\begin{equation}\label{def:B}
\|f\|_{|\cH|}^2:=\int_{0}^{1}\int_{0}^{1}|f(s) f(t)||s-t|^{2H-2}\dd s\dd t<\infty.
\end{equation}

\vspace{-0.15cm}
\noindent Clearly, $|\cH|$ is a Banach space and is dense in $\cH$.
\vspace{-0.2cm}
 
\begin{remark}
The inner product defined in \eqref{def:H_inner} is indeed positive definite since for $H\in(1/2,1)$,
    \vspace{-0.1cm}
    \begin{equation*}
        \begin{aligned}
\left<f,f\right>_{\mathcal{H}}&=\int_{0}^{1}\int_{0}^{1}f(s)f(t)|s-t|^{2H-2}\dd s\dd t\\&=\int_{0}^{1}\int_{0}^{1}f(s)f(t)\left(C_H\int_{\bR}|s-r|^{H-\frac{3}{2}}|t-r|^{H-\frac{3}{2}}\dd r\right)\dd s\dd t\\&=C_H\int_{\bR}\left(\int_{0}^{1}f(s)|s-r|^{H-\frac{3}{2}}\dd s\right)^2\dd r\ge 0.
        \end{aligned}
    \end{equation*}
    Also note that 
    $$\|f\|_{L^1(\bR)}^2\leq \int_{0}^{1}\int_{0}^{1}|f(s) f(t)||s-t|^{2H-2}\dd s\dd t\leq C_H\|f\|_{L^{1/H}(\bR)}^2, $$
   where the second inequality is due to Lemma \ref{lem:1/H norm}. Hence, we have $L^{1/H}\subset|\cH|\subset L^1$. For more details, we refer to Section 5.1.3 in \cite{Nua06}.
\end{remark}

Let $\{W(f),f \in \cH \}$ be an isonormal Gaussian process with covariance $\E[W(f) W(g)]=\la f, g\ra_\cH.$  The \emph{Wiener integral} $W(f)$ is often denoted in the integral form:
\vspace{-0.1cm}
\[
\int_0^1f(s)W(\dd s):= W(f). 
\]

\vspace{-0.15cm}
\noindent For $n\in \bN_0$, the $n$-th {\it Hermite polynomial} is given by
\vspace{-0.1cm}
\begin{equation*}
H_n(x):= (-1)^n\mathrm{e}^{x^2/ 2}\frac{\dd^n}{\dd x^n}\mathrm{e}^{-x^2/ 2}, \, x\in \bR.
\end{equation*}

\vspace{-0.15cm}
\noindent For  $f\in \cH$ with $\|f\|_{\cH}=1$,  the \emph{multiple Wiener integral} of $f^{\otimes n}$ is defined by  
\vspace{-0.1cm}
\begin{equation}\label{Wiener}
    \I_n(f^{\otimes n}):=H_n(W(f)),
\end{equation}
 where $f^{\otimes n}$ is a function on $[0,1]^n$ defined by $f^{\otimes n}(t_1,\cdots,t_n):=\prod_{k=1}^n f(t_k)$.

As in \eqref{def:H_inner} and \eqref{def:B}, let the Hilbert space ${\cH}^{ \otimes k}$ be the completion of smooth functions with compact support under the inner product:
\vspace{-0.1cm}
\begin{equation}\label{def:H_dim_inner}
\begin{aligned}
\left<f,g\right>_{\mathcal{H}^{\otimes k}}&:=\int_{[0,1]^{2k}}f(s_1,\cdots,s_k)g(t_1,\cdots,t_k)\prod_{i=1}^k|s_i-t_i|^{2H-2}\dd \boldsymbol{s}\dd \boldsymbol{t},
\end{aligned}
\end{equation}
and we also denote by $|\cH|^{\otimes k}$ the linear subspace of measurable functions $f$ on $[0,1]^k$ satisfying
\vspace{-0.1cm}
\begin{equation}\label{def:B_HD}
\|f\|_{|\cH|^{\otimes k}}^2:=\int_{[0,1]^{2k}}|f(s_1,\cdots,s_k)g(t_1,\cdots,t_k)|\prod_{i=1}^k|s_i-t_i|^{2H-2}\dd \boldsymbol{s}\dd \boldsymbol{t}<\infty.
\end{equation}

\vspace{-0.15cm}
\noindent Let ${\cH}^{\hat \otimes k}$ be the subspace of ${\cH}^{\otimes k}$ containing only symmetric functions. For  $f\in {\cH}^{\hat\otimes k}$,  the $k$-th \emph{multiple Wiener integral} $\I_k(f)$ can be defined by \eqref{Wiener} via a limiting argument, and moreover,
\vspace{-0.1cm}
$$
{\mathbb E}[\I_m(f)\I_n(g)]=n!\langle f,g\rangle_{\cH^{\otimes n}}\mathbf1_{\{m=n\}}, ~ {\rm for }~ f \in {\cH}^{\hat\otimes m}, g\in {\cH}^{\hat\otimes n}.
$$

\vspace{-0.15cm}
\noindent For general $f\in {\cH}^{\otimes k}$, we simply set $\I_k(f):= \I_k(\hat f),$ where $\hat  f$ denotes the symmetrization of $f$.  We also use the following integral form for multiple Wiener integrals:
\vspace{-0.1cm}
\begin{equation*}
\int_{[0,1]^n}f({\mathbf t}) W(\dd {t_1})\diamond\cdots \diamond W(\dd {t_n})=\int_{[0,1]^n}f({\mathbf t})\prod_{i=1}^n \diamond W(\dd {t_i}):=    \I_n(f).
\end{equation*}

\vspace{-0.15cm}
For  $f\in \cH^{\otimes n}$ and $g\in \cH^{\otimes m}$, the  \emph{Wick product} of $\I_n(f)$ and $\I_m(g)$ is defined by
\vspace{-0.1cm}
\begin{equation}\label{e:wick}
    \I_n(f) \diamond \I_m(g) := \I_{n+m}(f\otimes g). 
\end{equation}

\vspace{-0.15cm}
\noindent For square integrable  random variables $F$ and $G$ with \emph{Wiener chaos} expansions $F=\sum_{k=0}^\infty \I_k(f_k)$ and $G=\sum_{k=0}^\infty \I_k(g_k)$ respectively, we have $F\diamond G=\sum_{n=0}^\infty \sum_{m=0}^\infty \I_{n+m}(f_n \otimes g_m)$, 
as long as  the right-hand side term  is well defined (i.e., the series converges in $L^2$). For a centered Gaussian random variable $F$, its \emph{Wick exponential} is defined by
\vspace{-0.1cm}
\begin{equation}\label{e:wicke}
\exp^\diamond(F):=\exp\left(F-\frac{\bE[F^2]}{2}\right)=\sum_{m=0}^\infty\frac{1}{m!}F^{\diamond m}.
\end{equation}

\vspace{-0.15cm}
The \emph{multiple Stratonovich
 integral} of $g^{\otimes n}$ for $g\in \cH$ is defined by $ \mathbb I_n(g^{\otimes n}):=W(g)^n,$
and $\mathbb I_n(g)$ for $g\in \cH^{\hat{\otimes}  n}$ by  taking limits for linear combinations (see \cite{Hu16} for more details).  We also take the following notation for $\mathbb I_n(f)$
\vspace{-0.1cm}
\begin{align*}
\int_{[0,1]^n}f({\mathbf t}) W(\dd {t_1})\cdots W(\dd {t_n})
=\int_{[0,1]^n}f({\mathbf t})\prod_{i=1}^n W(\dd {t_i}):=\mathbb I_n(f). 
\end{align*}

\vspace{-0.15cm}
\noindent For $f\in\cH^{\otimes n}$ and $k=0,1,\cdots,\lfloor\frac{n}{2}\rfloor$, define the $k$-th order \emph{trace} $\mathrm{Tr}^k f$ of $f$ by
\vspace{-0.1cm}
\begin{equation}
\begin{aligned}
\mathrm{Tr}^k f(t_1,...,t_{n-2k})&\overset{\text{def}}{=}\int_{[0,1]^{2k}}f(s_1,...,s_{2k},t_1,...,t_{n-2k})\prod\limits_{i=1}^k K(s_{2i}-s_{2i-1})\dd s_{2i-1}\dd s_{2i},
\label{e:trace}
\end{aligned}
\end{equation}

\vspace{-0.15cm}
\noindent where $K(x-y)$ is given in \eqref{def:noise}. Multiple Stratonovich integrals $\mathbb{I}_m(f)$  and multiple Wiener integrals $\I_m(f)$ are related via the following celebrated  Hu--Meyer formula:
\vspace{-0.1cm}
\begin{equation}
\begin{aligned}
\mathbb{I}_n(f)=\sum_{k=0}^{[\frac{n}{2}]}\frac{n!}{k!(n-2k)!2^k}\I_{n-2k}(\mathrm{Tr} ^k \hat{f}),\quad\text{for}~f \in \cH^{\otimes n},
\label{H-M}
\end{aligned}
\end{equation}

\vspace{-0.15cm}
\noindent as long as $\interleave f\interleave_n<\infty$, where
\vspace{-0.1cm}
\begin{equation}
\begin{aligned}
\interleave f\interleave^2_{n}:=\mathbb{E}[|\mathbb{I}_n(f)|^2]=\sum_{k=0}^{[\frac{n}{2}]}\frac{1}{(n-2k)!}\left(\frac{n!}{k!2^k}\right)^2 \Vert \mathrm{Tr} ^k \hat{f}\Vert^2_{\cH^{\otimes (n-2k)}}.
\label{def-sm}
\end{aligned}
\end{equation}

\subsection{U-statistics}

\begin{definition}
We denote
\vspace{-0.1cm}
$$\mathcal{R}_{N}^{m}:=\Big\{\Big(\frac{\boldsymbol{i}-\vec{1}_{m}}{N},\frac{\boldsymbol{i}}{N}\Big]:\boldsymbol{i}\in\llbracket N \rrbracket^{m},\vec{1}_{m} \text{ is  the $m$-dimensional vector of all ones}\Big\}.$$

\vspace{-0.15cm}
\noindent Given $f\in|\cH|^{\otimes m}$, we denote by $\mathcal A_N(f)$ its conditional expectation with respect to the $\sigma$-field generated by $\mathcal R_N^m$, i.e., 
\vspace{-0.1cm}
\begin{equation}\label{e:AN}
\mathcal{A}_{N}(f)(\boldsymbol{t})=\sum_{B\in\mathcal{R}_{N}^{m}}\frac{1}{|B|}\int_{B}f(\boldsymbol{s})d\boldsymbol{s}\cdot \mathrm{I}_{B}(\boldsymbol{t}),
\end{equation}

\vspace{-0.15cm}
\noindent where $|B|$ is the Lebesgue measure of B.
\end{definition}
We introduce the \emph{U-statistics} for our model, which appears in the chaos expansion for the Wick-ordered partition function \eqref{e:partition'}. 
\begin{definition} \label{def:Im^N}
For $f\in|\cH|^{\otimes m}$, let
\vspace{-0.1cm}
\begin{equation}\label{e:I(f)}
\begin{aligned}
\I_{m}^{(N)}(f)&:=N^{-mH}\sum_{\boldsymbol{n}\in\llbracket N \rrbracket^{m}}\mathcal{A}_{N}(f)\left(\frac{n_{1}}{N},\cdots,\frac{n_{m}}{N}\right)\omega_{n_{1}}\diamond\cdots\diamond\omega_{n_{m}}=:N^{-mH}\sum_{\boldsymbol{n}\in\llbracket N \rrbracket^{m}}\mathcal{A}_{N}(f)(\boldsymbol{t})\,\omega_{\boldsymbol{n}}^\diamond,
\end{aligned}
\end{equation}

\vspace{-0.15cm}
\noindent where we denote
\vspace{-0.1cm}
\begin{equation}\label{e:wt}
\omega_{\boldsymbol{n}}^\diamond:=\omega_{n_{1}}\diamond\cdots\diamond\omega_{n_{m}} \quad\text{ and }\quad \boldsymbol{t}=(t_{1}\cdots t_{m}):=\frac{\boldsymbol{n}}{N}=\left(\frac{n_{1}}{N},\cdots,\frac{n_{m}}{N}\right).
\end{equation}
\end{definition}

\vspace{-0.3cm}
\begin{lemma}\label{lem:jensen}
For $m,N\in\bN $ and $f\in|\cH|^{\otimes m}$, let $\mathcal{A}_{N}(f)$ be defined as above. Then we have,
\vspace{-0.1cm}
$$\|\mathcal{A}_{N}(f)\|_{|\cH|^{\otimes m}}\leq C^{m}\|f\|_{|\cH|^{\otimes m}},$$

\vspace{-0.15cm}
\noindent where C is a constant depending only on H.
\end{lemma}
\begin{proof}
The desired result  follows directly from \cite[Lemma C.1]{rsw24}.
\end{proof}

\begin{lemma}\label{lem:isometry}
There exists a constant $C$ depending only on $H$  such that for all $f\in|\cH|^{\otimes m}$ and $N\in\bN$,
\vspace{-0.1cm}
$$\bE\Big[\Big(\I_{m}^{(N)}(f)\Big)^{2}\Big]\leq C^{m}m!\|\hat{f}\|_{|\cH|^{\otimes m}}^{2},$$

\vspace{-0.15cm}
\noindent where $\hat{f}$ is the symmetrization of f.
\end{lemma}
\begin{proof}
For $\boldsymbol{n},\boldsymbol{n}'\in\llbracket N \rrbracket^{m}$, by Lemma \ref{lem:Wick}, we have $\bE\left[\omega_{\boldsymbol{n}}^\diamond~\omega_{\boldsymbol{n'}}^\diamond\right]=\sum_{\sigma\in\mathcal{P}_{m}}\prod_{i=1}^{m}\gamma\big(n_{i}-n'_{\sigma(i)}\big)$, where  $ \mathcal{P}_{m}$ denotes the set of all permutations on $\llbracket m\rrbracket$. Thus, we have
\vspace{-0.1cm}
\begin{equation}\nonumber
\begin{aligned}
\bE\Big[\Big(\I_{m}^{(N)}(f)\Big)^{2}\Big]&=\bE\Big[\Big(\I_{m}^{(N)}(\hat{f})\Big)^{2}\Big]=N^{-2mH}\sum_{\boldsymbol{n},\boldsymbol{n'}\in\llbracket N \rrbracket^{m}}\mathcal{A}_{N}(\hat{f})(\boldsymbol{t})\mathcal{A}_{N}(\hat{f})(\boldsymbol{t'} )\sum_{\sigma\in\mathcal{P}_{m}}\prod_{i=1}^{m}\gamma\big(n_{i}-n'_{\sigma(i)}\big)\\
&=m!N^{-2mH}\sum_{\boldsymbol{n},\boldsymbol{n'}\in\llbracket N \rrbracket^{m}}\mathcal{A}_{N}(\hat{f})(\boldsymbol{t})\mathcal{A}_{N}(\hat{f})(\boldsymbol{t'})\prod_{i=1}^{m}\gamma(n_{i}-n'_{i}),
\end{aligned}
\end{equation}

\vspace{-0.15cm}
\noindent where the last equality follows from the symmetry of $\mathcal A_N(\hat f)$. 

For $t\in\bR$, let
\vspace{-0.2cm}
\begin{equation}\label{def:gamma_N}
\gamma_{N}(t):=N^{2-2H}\gamma([|t|N]).
\end{equation}
By \eqref{con-gamma} and \eqref{con-gamma'}, we have  $\gamma_{N}\lesssim |t|^{2H-2}$ and $\lim_{N\rightarrow\infty}\gamma_{N}=|t|^{2H-2}$. Therefore, we get by a Riemann sum approximation,
\vspace{-0.2cm}
\begin{equation}\nonumber
\begin{aligned}
\bE\Big[\left(\I_{m}^{(N)}(f)\right)^{2}\Big]&=m!N^{-2m}\sum_{\boldsymbol{n},\boldsymbol{n'}\in\llbracket N \rrbracket^{m}}\mathcal{A}_{N}(\hat{f})(\boldsymbol{t})\mathcal{A}_{N}(\hat{f})(\boldsymbol{t'})\prod_{i=1}^{m}N^{2-2H}\gamma(n_{i}-n'_{i})\\
&\lesssim m!\int_{[0,1]^{2m}}|\mathcal{A}_{N}(\hat{f})(\boldsymbol{t})\mathcal{A}_{N}(\hat{f})(\boldsymbol{t'})|\prod\limits_{i=1}^m|t_{i}-t_{i}'|^{2H-2}\dd t_{i}\dd t_{i}'\\
&=m!\|\mathcal{A}_{N}(\hat{f})\|_{|\cH|^{\otimes m}}^{2}\leq m! C^{m}\|\hat{f}\|_{|\cH|^{\otimes m}}^{2},
\end{aligned}
\end{equation}

\vspace{-0.15cm}
\noindent where the last inequality follows from Lemma \ref{lem:jensen}.
\end{proof}

\begin{proposition}\label{prop:1-joint}
For $f\in|\cH|$,  $\I_{1}^{(N)}(f)$ converges weakly to the Wiener integral $\I_{1}(f)$ as $N\rightarrow\infty$. Moreover, for any $k\in\bN$\ and $f_{1},\cdots,f_{k}\in|\cH|$, we have the joint convergence in distribution:
\vspace{-0.2cm}
$$\Big(\I_{1}^{(N)}(f_{1}),\cdots,\I_{1}^{(N)}(f_{k})\Big)\overset{(d)}{\longrightarrow}\left(\I_{1}(f_{1}),\cdots,\I_{1}(f_{k})\right),\quad\text{as}~N\rightarrow\infty.$$

\vspace{-0.5cm}
\end{proposition}

\begin{proof}
First we consider indicator functions of the form $f(x)=\mathrm{I}_{(s,t]}(x)$ with $s,t\in\Omega_{N}$. Then
\vspace{-0.2cm}
$$\mathcal{A}_N(f)(r)=\mathrm{I}_{(s,t]}(r)\quad\text{and}\quad\I_{1}^{(N)}(f)=N^{-H}\sum_{n=1}^{N}\omega_{n}\mathrm{I}_{(Ns,Nt]}(n),$$

\vspace{-0.25cm}
\noindent and we have
\vspace{-0.2cm}
\begin{equation}\nonumber
\begin{aligned}
\bE_{\omega}\Big[\Big(\I_{1}^{(N)}(f)\Big)^{2}\Big]&=\bE_{\omega}\Big[N^{-2H}\Big(\sum_{n=1}^{N}\omega_{n}\I_{(Ns,Nt]}(n)\Big)^{2}\Big]=N^{-2H}\sum_{n,n'\in(Ns,Nt]}\gamma(n-n'). 
\end{aligned}
\end{equation}

\vspace{-0.25cm}
\noindent Denoting $r=\frac{n}{N}, r'=\frac{n'}{N}$, we get
\vspace{-0.2cm}
$$\bE\Big[\Big(\I_{1}^{(N)}(f)\Big)^{2}\Big]=N^{-2}\sum_{r,r'\in(s,t]\cap\Omega_{N}}\gamma_{N}(r-r')\overset{N\rightarrow\infty}{\longrightarrow}\int_{s}^{t}\int_{s}^{t}|r-r'|^{2H-2}\dd r\dd r'=\|f\|_{\mathcal{H}}^{2}.$$

\vspace{-0.25cm}
\noindent Noting that $\I_{1}^{(N)}(f) $ and $\I_{1}(f)$\ are Gaussian random variables and $\I_{1}(f)$ has zero mean and variance $\|f\|_{\mathcal{H}}^{2} $, we have $\I_{1}^{(N)}(f)\overset{(d)}{\longrightarrow}\I_{1}(f)$ as $N\rightarrow\infty.$

Similarly, one can show that $\I_{1}^{(N)}(f)\overset{(d)}{\longrightarrow}\I_{1}(f)$ as $N\rightarrow\infty$ also holds for  a simple function $f$. For a general function $f\in|\cH|$, since simple functions are dense in $|\cH|$, there exists a sequence of simple functions $\{f_{n}\}_{n\in\bN}$ converging to $f$ in $|\mathcal{H}|$. On the one hand, for each fixed $n$, we have shown that $\I_{1}^{(N)}(f_{n})\overset{(d)}{\longrightarrow}\I_{1}(f_{n})$ as $N\rightarrow\infty.$ On the other hand, for each fixed $N$, by Lemma \ref{lem:isometry}, we have $\I_{1}^{(N)}(f_{n})\rightarrow\I_{1}^{(N)}(f)$ as $n\rightarrow\infty$ in $L^{2}$  uniformly in $N$. Meanwhile, we also have $\I_{1}(f_{n})\rightarrow\I_{1}(f)$ as $n\rightarrow\infty$\ in $L^{2}$ by the property of the Wiener integral. Combining the above convergences and by Lemma~\ref{lem:weak-con}, we have $\I_{1}^{(N)}(f)\overset{(d)}{\longrightarrow}\I_{1}(f)$\ as $N\rightarrow\infty$ for $f\in|\cH|$. Finally, the joint convergence follows from the linearity of $\I_{1}^{(N)}$ and the Cram\'er-Wold Theorem.
\end{proof}

\begin{proposition}\label{prop:2-joint}
For any $m\in\bN$ and $f\in|\cH|^{\otimes m} $, we have
\vspace{-0.2cm}
\begin{equation}\label{eq:multi-conv}
\I_{m}^{(N)}(f)\overset{(d)}{\longrightarrow}\I_{m}(f),\quad\text{as}~N\rightarrow\infty,
\end{equation}

\vspace{-0.2cm}
\noindent where $ \I_{m}(f)$ is the $m$-th multiple Wiener integral of f. Moreover, for any $k\in\bN$ and $f_{1},\cdots,f_{k}$ with $f_{i}\in|\cH|^{\otimes l_{i}}$ for $l_{i}\in\bN$, we have the joint convergence in distribution
\vspace{-0.1cm}
\begin{equation}\label{eq:joint_weak_conv}
\left(\I_{l_{1}}^{(N)}(f_{1}),\cdots,\I_{l_{k}}^{(N)}(f_{k})\right)\overset{(d)}{\longrightarrow}\left(\I_{l_{1}}(f_{1}),\cdots,\I_{l_{k}}(f_{k})\right),\quad\text{as}~N\rightarrow\infty.
\end{equation}

\vspace{-0.5cm}
\end{proposition}
\begin{proof}
We only prove \eqref{eq:multi-conv}. The joint convergence follows from the linearity of $\I_{l}^{(N)}$ and the Cram\'er-Wold Theorem.

First, we consider $f=h_{1}^{\otimes p_{1}}\otimes\cdots\otimes h_{l}^{\otimes p_{l}}$ with $h_{i}\in|\cH|$ and $p_{1}+\cdots+p_{l}=m$. For such $f$, we have
\vspace{-0.1cm}
\begin{equation}\nonumber
\begin{aligned}
\I_{m}^{(N)}(f)&=N^{-2mH}\sum_{\boldsymbol{n}\in\llbracket N \rrbracket^{m}}\mathcal{A}_{N}(h_{1}^{\otimes p_{1}}\otimes\cdots\otimes h_{l}^{\otimes p_{l}})\omega_{\boldsymbol{n}}^\diamond\\
&=N^{-2mH}\sum_{\boldsymbol{n}\in\llbracket N \rrbracket^{m}}\mathcal{A}_{N}(h_{1})^{\otimes p_{1}}\otimes\cdots\otimes\mathcal{A}_{N}(h_{l})^{\otimes p_{l}}\omega_{\boldsymbol{n}}^\diamond=\left(\I_{1}^{(N)}(h_{1})\right)^{\diamond p_{1}}\diamond \cdots\diamond\left(\I_{1}^{(N)}(h_{l})\right)^{\diamond p_{l}}
\end{aligned}
\end{equation}

\vspace{-0.15cm}
\noindent By Lemma \ref{lem:isometry}, we have $\bE[(\I_{1}^{(N)}(h_{i}))^{2}]\leq C\|h_{i}\|_{|\cH|}^{2}$ uniformly in $N$. Note that $\I_{1}^{(N)}(h_{i}) $ is a centered Gaussian random variable, and thus  its higher moments can be expressed via its variance. Hence, we have  for any $q\in\bN$,
\vspace{-0.1cm}
$$\bE\Big[\Big(\I_{1}^{(N)}(h_{i})\Big)^{2q}\Big]\leq C^q(2q-1)!!\|h_{i}\|_{|\cH|}^{2q}. $$

\vspace{-0.15cm}
Combining the joint convergence in Proposition \ref{prop:1-joint}, we can apply \cite[Lemma B.3]{rsw24} to get
\vspace{-0.1cm}
\begin{equation}\nonumber
\begin{aligned}
\I_{m}^{(N)}(f)&=\left(\I_{1}^{(N)}(h_{1})\right)^{\diamond p_{1}}\diamond \cdots\diamond\left(\I_{1}^{(N)}(h_{l})\right)^{\diamond p_{l}}\overset{(d)}{\longrightarrow}\left(\I_{1}(h_{1})\right)^{\diamond p_{1}}\diamond\cdots\diamond\left(\I_{1}(h_{l})\right)^{\diamond p_{l}},\quad\text{as}~N\rightarrow\infty,
\end{aligned}
\end{equation}

\vspace{-0.15cm}
\noindent where  the weak limit is equal to, noting \eqref{e:wick},
\vspace{-0.1cm}
\begin{equation*}
\left(\I_{p_{1}}(h_{1}^{\otimes p_{1}})\right)\diamond\cdots\diamond\left(\I_{p_{l}}(h_{l}^{\otimes p_{l}})\right)=\I_{m}\left(h_{1}^{\otimes p_{1}}\otimes\cdots\otimes h_{l}^{\otimes p_{l}}\right)=\I_{m}(f).
\end{equation*}

\vspace{-0.15cm}
For a general function $f\in|\cH|^{\otimes m}$, we can find a sequence of functions $\{f_{n}\}_{n\in\bN}$ which are the linear combinations of the functions of the form $h_{1}^{\otimes p_{1}}\otimes\cdots\otimes h_{l}^{\otimes p_{l}} $ satisfying $f_{n}\rightarrow f$ in $|\cH|^{\otimes m}$. Similarly to the proof of Proposition \ref{prop:1-joint}, we have
$\I_{m}^{(N)}(f)\overset{(d)}{\longrightarrow}\I_{m}(f)\text{, as}~N\rightarrow\infty$.
\end{proof}

\section{The Skorohod case}\label{sec:skor}
In Section \ref{subsec:weak_skorohod}, we study the weak convergence for the Wick-ordered partition function defined in \eqref{e:partition'} and prove Theorem \ref{thm1} which guarantees the weak convergence of the rescaled Wick-ordered partition functions in the $L^2$-regime when $H\in(\frac12, 1)$ and $\alpha\geq\frac12$.   In Section~\ref{sec:discussion},  we   discuss the uniform integrability and characterize the weak limit of the Wick-ordered partition functions for the case $\alpha\le \frac12$.

By a Taylor expansion and Wick exponential \eqref{e:wicke}, we get
\vspace{-0.1cm}
\begin{equation}\label{nzf_expand}
\begin{aligned}
\nzf&=\bE_{\tau}\Big[\exp\Big\{h\sum_{x\in\Omega_{N}}\delta_{x}\Big\}\exp\Big\{\diamond\Big(\beta\sum_{x\in\Omega_{N}}\omega_{Nx}\delta_{x}\Big) \Big\}\Big]\\
&=\bE_{\tau}\Big[\Big(\sum_{k=0}^{\infty}\frac{1}{k!}\Big(h\sum_{x\in\Omega_{N}}\delta_{x}\Big)^{k}\Big)\Big(\sum_{m=0}^{\infty}\frac{1}{m!}\Big(\beta\sum_{x\in\Omega_{N}}\omega_{Nx}\delta_{x}\Big)^{\diamond m}\Big)\Big]\\
&=\sum_{k=0}^{\infty}\sum_{m=0}^{\infty}\frac{h^{k}\beta^{m}}{k!m!}\sum_{\boldsymbol{n}\in\llbracket N\rrbracket^{k+m}}\bP(n_{1}\in\tau,n_{2}\in\tau,\cdots,n_{k+m}\in\tau)\omega_{n_{1}}\diamond\omega_{n_{2}}\diamond\cdots\diamond\omega_{n_{m}}.
\end{aligned}
\end{equation}


\vspace{-0.15cm}
We shall write $\tilde{Z}_N:=\tilde{Z}_{N,\beta_N}^{\omega,h_N}$ throughout the rest of the paper for simplicity. By \eqref{nzf_expand}, we have
\vspace{-0.1cm}
\begin{equation}\label{e:tZN}
\tilde{Z}_N=\sum_{k=0}^{\infty}\sum_{m=0}^{\infty}\frac{1}{k!m!}S_{m,k}^{(N)},
\end{equation}

\vspace{-0.15cm}
\noindent where
\begin{equation}\label{e:S}
S_{m,k}^{(N)}:=\beta_{N}^{m}h_{N}^{k}\sum_{\boldsymbol{n}\in\llbracket N\rrbracket^{k+m}}\bP\big(n_{1}\in\tau,n_{2}\in\tau,\cdots,n_{k+m}\in\tau\big)\,\omega_{n_{1}}\diamond\omega_{n_{2}}\diamond\cdots\diamond\omega_{n_{m}}.
\end{equation}

\vspace{-0.1cm}
By recalling the definition of $\beta_N$ and $h_N$ from \eqref{scale:beta}, denote
\vspace{-0.1cm}
\begin{equation}\label{def:f_mk}
\begin{split}
& f_{m,k}^{(N)}(t_{1},\cdots,t_{m})\\
&:=\begin{cases}
\displaystyle 
\sum_{\boldsymbol{q}\in\llbracket N \rrbracket^{m}}\sum_{\boldsymbol{s}\in\llbracket N \rrbracket^{k}}N^{-k}\frac{L(N)^{m+k}}{(N^{\alpha-1})^{m+k}}P_\tau(\boldsymbol{q},\boldsymbol{s})\prod_{i=1}^{m}\mathrm{I}_{(\frac{q_{i}-1}{N},\frac{q_{i}}{N}]}(t_{i}),\quad&\text{for}~1-H<\alpha<1,\\[10pt]
\displaystyle
\sum\limits_{\boldsymbol{q}\in\llbracket N \rrbracket^{m}}\sum\limits_{\boldsymbol{s}\in\llbracket N \rrbracket^{k}}N^{-k}P_\tau(\boldsymbol{q},\boldsymbol{s})\prod\limits_{i=1}^{m}\mathrm{I}_{(\frac{q_{i}-1}{N},\frac{q_{i}}{N}]}(t_{i}),\quad&\text{for}~\bE[\tau_1]<\infty,\\[10pt]
\displaystyle \sum\limits_{\boldsymbol{q}\in\llbracket N \rrbracket^{m}}\sum\limits_{\boldsymbol{s}\in\llbracket N \rrbracket^{k}}N^{-k}l(N)^{m+k} P_\tau(\boldsymbol{q},\boldsymbol{s})\prod\limits_{i=1}^{m}\mathrm{I}_{(\frac{q_{i}-1}{N},\frac{q_{i}}{N}]}(t_{i}), \quad&\text{if}~\alpha=1 \text{ and } \bE[\tau_1]=\infty,
\end{cases}
\end{split}
\end{equation}

\vspace{-0.15cm}
\noindent where
\vspace{-0.2cm}
\begin{equation}\label{def:P_tau}
P_\tau(\boldsymbol{q},\boldsymbol{s}):=\bP(\{q_{1},\cdots,q_{m},s_{1},\cdots,s_{k}\}\subset\tau).
\end{equation}
Then, recalling the definition \eqref{e:I(f)} for $\I_m^{(N)}(f)$, we have
\vspace{-0.1cm}
\begin{equation}\label{e:SN}
S_{m,k}^{(N)}=\hat{\beta}^m\hat{h}^k\I^{(N)}_m(f_{m,k}^{(N)}).
\end{equation}


\vspace{-0.5cm}

\subsection{Weak convergence} \label{subsec:weak_skorohod}

First, we prove the weak convergence of each chaos, i.e., the weak convergence of $\I_m^{(N)}(f_{m,k}^{(N)})$, where $\I_m^{(N)}$ is given by \eqref{e:I(f)} and $f_{m,k}^{(N)}$ by \eqref{def:f_mk}.

\vspace{-0.2cm}
\begin{assumptionp}{(A1)}\label{H}
Assume  $H\in(\frac12, 1)$ and $\alpha+H>1$. \end{assumptionp}
\vspace{-0.5cm}

\begin{proposition}\label{prop:weak-con-Sk}
Assume \ref{H}. Then, we have
\vspace{-0.1cm}
\begin{equation}\label{chaos_conv}
\I_{m}^{(N)}(f_{m,k}^{(N)})\overset{(d)}{\longrightarrow}\I_{m}(\psi_{m,k}),\quad\text{as}~N\to\infty.
\end{equation}
Moreover, for any $p\in\bN$\ and $m_{1},\cdots,m_{p},k_{1},\cdots,k_{p}\in\bN_0$, we have the following joint convergence in distribution
\vspace{-0.1cm}
\begin{equation}\label{joint_weak_conv}
\left(\I_{m_1}^{(N)}(f_{m_1,k_1}^{(N)}),\cdots,\I_{m_p}^{(N)}(f_{m_p,k_p}^{(N)}) \right)\overset{(d)}{\longrightarrow}\left( \I_{m_1}(\psi_{m_1,k_1}),\cdots,\I_{m_p}(\psi_{m_p,k_p})\right),\quad\text{as}~N\rightarrow\infty.
\end{equation}

\vspace{-0.15cm}
\noindent Here
\vspace{-0.2cm}
\begin{eqnarray}\label{e:psi}
\psi_{m,k}(t_{1},\cdots,t_{m}):=
\begin{cases}
\displaystyle \int_{[0,1]^{k}} \phi_{m+k} (t_{1},\cdots,t_{m+k})\dd t_{m+1}\cdots \dd t_{m+k},\quad&\text{if}~ 1-H<\alpha<1,\\
\displaystyle\Big(\frac{1}{\bE[\tau_{1}]}\Big)^{m+k}, & \text{if }~\bE[\tau_{1}]<\infty,\\[10pt]
1\, , &\text{if } \alpha=1 \text{ and } \E[\tau_1]=\infty,
\end{cases}
\end{eqnarray}

\vspace{-0.15cm}
\noindent where the integrand $\phi_{m+k}$ is the symmetric function given by \eqref{e:phi} with $\alpha<1$. 
\end{proposition}

\begin{proof}
We shall focus on the case $1-H<\al<1$ and the proofs for the other two cases are similar. 

Firstly, observe that $\psi_{m,k}(t_{1},\cdots,t_{m})$ belongs to the space $|\cH|^{\otimes m}$ due to the following lemma.
\begin{lemma}\label{lem:norm_degenerate_trace}
 There exists a constant $C=C_{\alpha,H}$ such that for all $m,k\in\bN$,
 \vspace{-0.1cm}
\begin{equation*}
\|\psi_{m,k}\|_{|\cH|^{\otimes m}}^2\leq C_H^m\|\phi_{m+k}\|_{L^{1/H}(\bR^{m+k})}^2\leq\frac{C^{m+k}[(m+k)!]^{2H}}{\Gamma\big((m+k)\frac{\alpha+H-1}{H}+1\big)^{2H}}.
\end{equation*}

\vspace{-0.5cm}
\end{lemma}
\begin{proof}[Proof of Lemma \ref{lem:norm_degenerate_trace}]
Recall \eqref{def:B_HD} and  apply Lemma \ref{lem:1/H norm} and Jensen's inequality,
\begin{equation*}
\begin{split}
&\|\psi_{m,k}\|_{|\cH|^{\otimes m}}^2\leq C_H^m\|\psi_{m,k}\|_{L^{1/H}(\bR^m)}^2\leq C_H^m\|\phi_{m+k}\|_{L^{1/H}(\bR^{m+k})}^2.
\end{split}
\end{equation*}
Then by permuting $t_1,\cdots,t_{m+k}$ and by Lemma \ref{lem:a direct calculate}, we obtain the desired bound.
\end{proof}

By Proposition \ref{prop:2-joint}, we have $\I_{m}^{(N)}(\psi_{m,k})\overset{(d)}{\longrightarrow}\I_{m}(\psi_{m,k})\text{, as}~N\rightarrow\infty.$ To prove \eqref{chaos_conv}, we denote
\vspace{-0.1cm}
\begin{equation*}
Y_{m,k}^{(N)}=\I_{m}^{(N)}(\psi_{m,k})-\I_{m}^{(N)}(f_{m,k}^{(N)})=\I_{m}^{(N)}(\psi_{m,k}-f_{m,k}^{(N)}),
\end{equation*}

\vspace{-0.15cm}
\noindent and then it suffices to show $Y_{m,k}^{(N)}\rightarrow 0$ in $L^{2}$ as $N\rightarrow\infty$. By Lemma \ref{lem:isometry} and noting the symmetry of $f_{m,k}^{(N)} $ and $\psi_{m,k}$,  we get  $\bE[(Y_{m,k}^{(N)})^{2}]\leq m!C^{m}\|f_{m,k}^{(N)}-\psi_{m,k}\|_{|\cH|^{\otimes m}}^{2}.$  
Thus, it remains to show $\lim_{N\to\infty}\|f_{m,k}^{(N)}-\psi_{m,k}\|_{|\cH|^{\otimes m}}^{2}=0.$ Define 
\vspace{-0.2cm}
\begin{equation}\label{def:g_mk}
g_{m,k}^{(N)}(t_{1},\cdots,t_{m+k}):=\sum_{\boldsymbol{n}\in\llbracket N \rrbracket^{m+k}}\frac{L(N)^{m+k}}{(N^{\alpha-1})^{m+k}}\bP\left(\{n_{1},\cdots,n_{m+k}\}\subset \tau\right)\prod_{i=1}^{m+k}\I_{(\frac{n_{i}-1}{N},\frac{n_{i}}{N}]}(t_{i}).
\end{equation}

\vspace{-0.15cm}
\noindent Recall \eqref{def:f_mk} and we can get
\vspace{-0.2cm}
\begin{equation}\label{e:f-g}
f_{m,k}^{(N)}(t_{1},\cdots,t_{m})=\int_{[0,1]^k}g_{m,k}^{(N)}(t_{1},\cdots,t_{m+k})\prod\limits_{i=1}^k\dd t_{m+i}.
\end{equation}

\vspace{-0.15cm}
\noindent By \eqref{local} and the renewal property, we have $\lim_{N\to\infty}g_{m,k}^{(N)}(t_{1},\cdots,t_{m+k})=\phi_{m+k}
(t_{1},\cdots,t_{m+k})$. 
By Karamata's representation theorem for slowly varying function, we have for any $\epsilon>0$, there exists a constant $A=A_{\epsilon}\in(0,\infty)$ such that
\vspace{-0.1cm}
\begin{equation}\label{Karamata}
    \frac{1}{A}\left(\frac{n}{m}\right)^{-\epsilon}\leq \frac{L(n)}{L(m)}\leq A\left(\frac{n}{m}\right)^{\epsilon},~\text{ for all } n,m\in\bN~\text{with}~m\leq n.
\end{equation}

\vspace{-0.15cm}
\noindent Also by (\ref{renewl theorem}), for a larger constant $A'\in(0,\infty)$, we have
\vspace{-0.1cm}
\begin{equation*}
\frac{C_\al}{A'(t-s)^{1-\al-\epsilon}}\leq \frac{L(N)}{N^{\al-1}}\bP_{\tau}\left(N(t-s)\in\tau\right)\leq \frac{A'C_\al}{(t-s)^{1-\al+\epsilon}},~ \text{ for all} ~t,s\in\Omega_N.
\end{equation*}

\vspace{-0.15cm}
\noindent Together with \eqref{con-gamma} and \eqref{con-gamma'}, we have that for $t_i\in (\frac{n_i-1}{N}, \frac{n_i}{N}], i=1, \dots, m+k$,
\vspace{-0.1cm}
\begin{equation}\label{ine:g-phi}
 g_{m,k}^{(N)}(t_1,\cdots,t_{m+k})\leq C^{m+k}\left(\phi_{m+k}(t_1,\cdots,t_{m+k})\right)^{\frac{1-\al+\epsilon}{1-\al}}.
\end{equation}

\vspace{-0.15cm}
\noindent By choosing $\epsilon$ sufficiently small such that $(\alpha-\epsilon)+H>1$, the term on the right-hand side above belongs to $L^{\frac{1}{H}}$, as has been shown in Lemma ~\ref{lem:norm_degenerate_trace}. Then, by the dominated convergence theorem, we have $\lim_{N\to\infty}\|g_{m,k}^{(N)}-\phi_{m+k}\|_{L^{1/H}(\bR^{m+k})}=0$. 

Finally, Jensen's inequality yields
\vspace{-0.2cm}
\begin{equation*}
\int_{[0,1]^{m}}\left|f_{m,k}^{(N)}-\psi_{m,k}\right|^{\frac{1}{H}}\dd t_{1}\cdots \dd t_{m}\leq\int_{[0,1]^{m+k}}\left|g_{m,k}^{(N)}-\phi_{m+k}\right|^{\frac{1}{H}}\dd t_{1}\cdots \dd t_{m+k},
\end{equation*}

\vspace{-0.25cm}
\noindent and thus by Lemma \ref{lem:1/H norm}, $\|f_{m,k}^{(N)}-\psi_{m,k}\|_{|\cH|^{\otimes m}}^{2}\leq C_H^m\|f_{m,k}^{(N)}-\psi_{m,k}\|_{L^{1/H}(\bR^{m})}^{2}\to 0$, as $N\to\infty$.

The joint convergence in distribution \eqref{joint_weak_conv} is then a direct consequence of \eqref{eq:joint_weak_conv} in Proposition \ref{prop:2-joint} and $Y_{m_i,k_i}^{(N)}\to0$ in $L^2$ as $N\to\infty$ above for each $i=1,\cdots, p$.
\end{proof}
Now we are ready to prove our first main result. 
\begin{proof}[Proof of Theorem \ref{thm1}]

\textbf{Case (1):} We assume \ref{H} and $\alpha>\frac12$, i.e., $H\in(\frac12, 1)$ and $\alpha>\frac12$. Note that
\vspace{-0.2cm}
\begin{equation}\nonumber
\begin{aligned}
\tilde{Z}_{\hat{\beta },\hat{h }}&=\sum_{k=0}^{\infty}\frac{1}{k!}\int_{[0,1]^{k}}\phi_{k}(t_{1},\cdots,t_{k})\prod\limits_{i=1}^k\diamond\left(\hat{\beta }W(\dd t_{i})+\hat{h }\dd t_{i}\right)\\
&=\sum_{k=0}^{\infty}\sum_{p=0}^k \frac{\hat{\beta }^{p}\hat{h }^{k-p}}{k!}{\binom{k}{p}}\I_{p}\bigg(\int_{[0,1]^{k-p}}\phi_{k}(t_{1},\cdots,t_{k})\dd t_{p+1}\cdots\dd t_{k}\bigg)\\
&\xlongequal{q=k-p}\sum_{k=0}^{\infty}\sum_{p+q=k}\frac{\hat{\beta }^{p}\hat{h }^{q}}{p!q!}\I_{p}(\psi_{p,q})=\sum_{m=0}^{\infty}\sum_{k=0}^{\infty}\frac{\hat{\beta }^{m}\hat{h }^{k}}{m!k!}\I_{m}(\psi_{m,k}),
\end{aligned}
\end{equation}

\vspace{-0.25cm}
\noindent where $\psi_{m,k}$ is given in \eqref{e:psi}. 

Recall that $\tilde{Z}_N=\tilde{Z}_{N,\beta_N}^{\omega,h_N}$ is given in \eqref{e:tZN} (see also \eqref{e:SN}):
\vspace{-0.2cm}
$$\tilde{Z}_{N}:=\sum_{m=0}^{\infty}\sum_{k=0}^{\infty}\frac{1}{m!k!}S_{m,k}^{(N)}= \sum_{m=0}^{\infty}\sum_{k=0}^{\infty}\frac{\hat{\beta }^{m}\hat{h }^{k}}{m!k!} \I_{m}^{(N)}(f_{m,k}^{(N)}). $$ 

\vspace{-0.25cm}
\noindent We define, recalling \eqref{def:f_mk} and \eqref{e:psi} for the definitions of $f_{m,k}^{(N)}$ and $\psi_{m,k}$ respectively,
\vspace{-0.2cm}
\begin{equation*}
\tilde Z_{N}^{(R)}:=\sum_{k+m\leq R}\frac{\hat{\beta }^{m}\hat{h }^{k}}{m!k!}\I_{m}^{(N)}(f_{m,k}^{(N)})\quad\text{and}\quad
\tilde Z^{(R)}:=\sum_{k+m\leq R}\frac{\hat{\beta }^{m}\hat{h }^{k}}{m!k!}\I_{m}(\psi_{m,k}).
\end{equation*}

\vspace{-0.25cm}
By \eqref{joint_weak_conv} in Proposition \ref{prop:weak-con-Sk} and the Cram\'{e}r-Wold Theorem, we have that $\tilde Z_{N}^{(R)}\overset{(d)}{\longrightarrow}\tilde Z^{(R)}$ as $N\to\infty$. To complete the proof, we will apply Lemma \ref{lem:weak-con}. Hence, it suffices to prove, as $R\rightarrow\infty$,  $\tilde Z^{(R)}\overset{L^2}{\longrightarrow}\tilde{Z}_{\hat{\beta },\hat{h }}$ and  $\tilde Z_{N}^{(R)}\overset{L^2}{\longrightarrow}\tilde{Z}_{N}$ uniformly in $N$,  which is done in the following two steps respectively.

\textbf{Step (1):}  We have 
\vspace{-0.2cm}
\begin{equation}\label{eq:L2_diff}
\begin{aligned}
\Big\|\tilde Z^{(R)}-\tilde{Z}_{\hat{\beta },\hat{h }}\Big\|_{L^{2}}&=\Big\|\sum_{k+m> R}\frac{\hat{\beta }^{m}\hat{h }^{k}}{m!k!}\I_{m}(\psi_{m,k})\Big\|_{L^2}\leq \sum_{k+m> R}\frac{\hat{\beta }^{m}|\hat{h }|^{k}}{m!k!}\Big\| \I_{m}(\psi_{m,k})\Big\|_{L^2}\\
&\leq C_H^{m+k}\sum_{k+m>R}\frac{\hat{\beta }^{m}|\hat{h }|^{k}}{\sqrt{m!}k!}\Vert\phi_{m+k}\Vert_{L^{1/H}(\bR^{m+k})}\\
&\leq\sum_{k+m>R}(C_{\hat{\beta},\hat{h}}^{\alpha,H})^{m+k}\frac{[(m+k)!]^{H}}{\sqrt{m!}k!\Gamma\big((m+k)\frac{\alpha+H-1}{H}+1\big)^{H}},
\end{aligned}
\end{equation}

\vspace{-0.25cm}
\noindent where the second and the last inequality is due to Lemma \ref{lem:norm_degenerate_trace}. By Stirling's formula, each term in the sum on the right-hand side is bounded above by
\vspace{-0.1cm}
\begin{equation}\label{eq:stirling_bd1}
C^{m+k}\frac{(m+k)^{H/2}}{\sqrt{m^{1/2}k}\left((m+k)\frac{\alpha+H-1}{H}\right)^{H/2}}\frac{(m+k)^{(m+k)H}}{ m^{m/2} k^{k}(m+k)^{(m+k)(\alpha+H-1)}}\leq C^{m+k}\frac{(m+k)^{(m+k)(1-\alpha)}}{m^{m/2} k^{k}},
\end{equation}

\vspace{-0.1cm}
\noindent  Note that for $a,b,c>0, (a+b)^c\leq(2\max\{a,b\})^c\leq 2^c(a^c+b^c)$. Hence, the right-hand side of \eqref{eq:stirling_bd1} is bounded above by:
\vspace{-0.1cm}
\begin{equation*}
C^{m+k}\Big(\frac{m^{(1-\alpha)m} m^{(1-\alpha)k} }{m^{m/2} k^{k}}+\frac{k^{(1-\alpha)k}k^{(1-\alpha)m}}{m^{m/2} k^{k}}\Big)
\end{equation*}

\vspace{-0.1cm}
\noindent which is summable in $m,k$ assuming $\alpha>\frac12$ (this is where the additional condition $\alpha>\frac12$ is used) by Lemma \ref{lem:series-mk}. Hence, we have proved $\tilde Z^{(R)}\overset{L^2}{\longrightarrow}\tilde{Z}_{\hat{\beta },\hat{h }}$ as $R\rightarrow\infty$.

\textbf{Step (2):} To prove the uniform convergence $\tilde{Z}_{N}^{(R)}\overset{L^2}{\longrightarrow}\tilde{Z}_N$ as $R\to \infty$, recalling (\ref{def:f_mk}) and (\ref{e:SN}),
\vspace{-0.1cm}
\begin{equation}\label{eq:step_2}
\begin{split}
\Big\|\tilde Z^{(R)}_{N}-\tilde{Z}_{N}\Big\|_{L^{2}}^{2}&=\bE\Big[\Big(\sum_{k+m>R}\frac{1}{m!k!}S_{m,k}^{(N)}\Big)^{2}\Big]\\&=\sum\limits_{k+m>R}\sum\limits_{k'+m>R}\frac{\hat{h}^{k+k'}\hat{\beta}^{2m}}{k!(k')!(m!)^2}\bE\left[\I_{m}^{(N)}(f_{m,k}^{(N)})\I_{m}^{(N)}(f_{m,k'}^{(N)})\right]\\
&\leq \sum\limits_{k+m>R}\sum\limits_{k'+m>R}\frac{|\hat{h}|^{k+k'}\hat{\beta}^{2m}}{k!(k')!(m!)^2}C^m m!\|f_{m,k}^{(N)}\|_{|\cH|^{\otimes m}}\|f_{m,k'}^{(N)}\|_{|\cH|^{\otimes m}}.
\end{split}
\end{equation}

\vspace{-0.3cm}
\noindent where the inequality follows from the Cauchy- Schwarz inequality and Lemma \ref{lem:isometry}. Recalling (\ref{def:g_mk}) and (\ref{e:f-g}), by Jensen's inequality and Lemma \ref{lem:1/H norm}, then \eqref{eq:step_2} is bounded above by
\vspace{-0.1cm}
\begin{equation*}
\begin{split}
&\sum\limits_{k+m>R}\sum\limits_{k'+m>R}\frac{C^{2m+k+k'}}{k!(k')!m!}\|g_{m,k}^{(N)}\|_{L^{1/H}(\bR^{m+k})}\|g_{m,k'}^{(N)}\|_{L^{1/H}(\bR^{m+k})}\\
\leq&\sum\limits_{k+m>R}\sum\limits_{k'+m>R}\frac{C^{2m+k+k'}}{k!(k')!m!}\Big\|\phi_{m+k}^{\frac{1-\al+\epsilon}{1-\al}}\Big\|_{L^{1/H}(\bR^{m+k})}\Big\|\phi_{m+k'}^{\frac{1-\al+\epsilon}{1-\al}}\Big\|_{L^{1/H}(\bR^{m+k})}\\
\leq&\sum_{k+m>R}\frac{C^{m+k}}{\sqrt{m!}k!}\Big\|\phi_{m+k}^{\frac{1-\al+\epsilon}{1-\al}}\Big\|_{L^{1/H}(\bR^{m+k})}\sum_{k'+m'>R}\frac{C^{m'+k'}}{\sqrt{m'!}k'!}\Big\|\phi_{m'+k'}^{\frac{1-\al+\epsilon}{1-\al}}\Big\|_{L^{1/H}(\bR^{m'+k'})}.
\end{split}
\end{equation*}

\vspace{-0.15cm}
\noindent where the first inequality follows from (\ref{ine:g-phi}). By the same estimate as in \eqref{eq:L2_diff} above, the right-hand side of the last inequality goes to 0 as $R\to \infty$ by choosing $\epsilon$ sufficiently small such that $(\alpha-\epsilon)+H>1$ and $\al-\epsilon>\frac{1}{2}$.

\textbf{Case (2):} We assume \ref{H} and $\alpha=\frac12$, i.e., $H\in(\frac12,1)$ and $\alpha=\frac12$. The proof is almost the same as in \textbf{Case (1)} and we indicate only the necessary modifications. First, noting that by the assumption $0<\inf_{n\in\bN}L(n)\leq\sup_{n\in\bN}L(n)<+\infty$ for this case, the exponent $\epsilon$ in~\eqref{Karamata} vanishes. Then by \eqref{eq:L2_diff}, \eqref{eq:stirling_bd1} and \eqref{eq:step_2}, it suffices to show that
\vspace{-0.15cm}
\begin{equation}\label{eq:double_sum}
\sum\limits_{m,k=0}^\infty\big(\hat{\beta}C_{H}\big)^m\big(|\hat{h}|C_{H}\big)^k\frac{[(m+k)!]^H}{\sqrt{m!}k!\Gamma((m+k)(\frac{2H-1}{2H})+1)^H}\leq\sum\limits_{m,k=0}^\infty\big(\hat{\beta}C_{H}\big)^m\big(|\hat{h}|C_{H}\big)^k\frac{(m+k)^{\frac{m+k}{2}}}{m^{m/2}k^k}
\end{equation}

\vspace{-0.1cm}
\noindent converges for  $\hat{\beta}>0,|\hat{h}|>0$ small enough.

Note that
\vspace{-0.15cm}
\begin{equation*}
\frac{(m+k)^{\frac{m+k}{2}}}{m^{m/2}k^k}\leq\frac{(2k)^{k}}{m^{m/2}k^k}\mathrm{I}_{\{k\geq m\}}+\frac{(2m)^{\frac{m+k}{2}}}{m^{m/2}k^k}\mathrm{I}_{\{k\leq m\}}\leq\frac{2^k}{m^{m/2}}\mathrm{I}_{\{k\geq m\}}+\frac{2^m m^{\frac{k}{2}}}{k^k}\mathrm{I}_{\{k\leq m\}}.
\end{equation*}

\vspace{-0.1cm}
\noindent Then the right-hand side of \eqref{eq:double_sum} is bounded by
\vspace{-0.15cm}
\begin{equation*}
\begin{split}
&\sum\limits_{m=0}^\infty\frac{(\hat{\beta}C_H)^m}{m^{m/2}}\sum\limits_{k=m}^\infty(2|\hat{h}|C_H)^k+\sum\limits_{m=0}^\infty(2\hat{\beta}C_H)^m\sum\limits_{k=0}^m\frac{(C_H|\hat{h}|\sqrt{m})^{k}}{k!}\\
\leq&\sum\limits_{m=0}^\infty\frac{(\hat{\beta}C_H)^m}{m^{m/2}}\cdot\frac{1}{1-2|\hat{h}|C_H}+\sum\limits_{m=0}^\infty e^{C_H|\hat{h}|\sqrt{m}}(2C_H\hat{\beta})^m<+\infty,
\end{split}
\end{equation*}

\vspace{-0.1cm}
\noindent given that $2C_H|\hat{h}|<1$ and $2C_H\hat{\beta}<1$.

\textbf{Case (3):} We assume $\alpha<\frac12$. Denote by $\tau'$ an independent copy of $\tau$ and we have that
\vspace{-0.15cm}
\begin{equation*}
\begin{split}
\bE\Big[\tilde{Z}_N^2\Big]&=\bE_{\tau,\tau'}\Big[\exp\Big\{\beta_N^2\sum\limits_{n,m=1}^N\gamma(n-m)\mathbf{1}_{\{n\in\tau\}}\mathbf{1}_{\{m\in\tau'\}}+h_N\sum\limits_{n=1}^N\big(\mathbf{1}_{\{n\in\tau\}}+\mathbf{1}_{\{n\in\tau'\}}\big)\Big\}\Big]\\
&\geq\exp\Big\{\beta_N^2\sum\limits_{n,m=1}^N\gamma(n-m)-2N|h_N|\Big\}\bP(\{1,\cdots,N\}\subset\tau\cap\tau')\\
&\geq\exp\Big(C\hat{\beta}^2\frac{1}{N^{2(\alpha+H-1+\epsilon)}}\sum\limits_{n,m=1}^N|n-m|^{2H-2}-C|\hat{h}|N^{1-\alpha+\epsilon}\Big)(\bP(\tau_1=1))^{2N}\\
&\geq\exp\Big(C\Big(N^{2(1-\alpha-\epsilon)}-N^{1-\alpha+\epsilon}\Big)\Big)(\bP(\tau_1=1))^{2N},
\end{split}
\end{equation*}

\vspace{-0.15cm}
\noindent which tends to infinity as $N\to\infty$ if $\alpha<\frac12$, by choosing $\epsilon>0$ small enough in \eqref{Karamata}. 
\end{proof}

\begin{remark}\label{remark:alpha}

We now comment on the case $\alpha=\tfrac12$, which represents the threshold between the $L^2$ regime and the non-$L^2$ regime. However, this should not be regarded as the marginal case between the disorder-relevant and disorder-irrelevant regimes, since in Section~\ref{sec:discussion} we show that the subsequential limit of $\{\tilde{Z}_N\}_N$ can be nonzero even for some $\alpha<\tfrac12$. Moreover, the case $\alpha=\tfrac12$ can still be treated within the framework of \cite{AKQ14,CSZ16}, where disorder is relevant, rather than that of \cite{CSZ17+}, where disorder is marginally relevant. These two frameworks are fundamentally different.

In comparison with the parabolic Anderson model on $\mathbb{R}^2$, where the $L^2$ solution exists only locally while the $L^1$ solution exists globally (see \cite{qrv}),  we believe that when $\alpha=\tfrac12$, the $L^2$ regime of our model exists \emph{only locally}  and the $L^1$ regime exists \emph{globally} (see Remarks~\ref{rmk:A},~\ref{rem:highlight}, and Section~\ref{sec:discussion}) We also note that our assumption on $L(\cdot)$ may not be optimal. To allow for a general slowly varying function $L(\cdot)$, one would need to verify the dominated convergence theorem carefully. Since our main purpose here is to highlight the critical nature of $\alpha=\tfrac12$ by a simple argument, we are content with the present restriction on $L(\cdot)$.
\end{remark}

\subsection{Some discussion on uniform integrability}\label{sec:discussion}

By recalling Section \ref{sec12}, the Weinrib--Halperin prediction suggests that our pinning model should be disorder relevant  under \ref{H}, i.e., $H\in(\frac12,1)$ and $\alpha+H>1$. This is partially confirmed by our results in the following sense: Proposition~\ref{prop:weak-con-Sk} yields  the weak convergence of each chaos  under \ref{H}; while for the weak convergence of the partition functions, an extra condition  $\alpha\geq\frac12$ is imposed  in order to guarantee the $L^2$-convergence of the Wiener chaos expansion (see Theorem \ref{thm1}). This implies that when $H\in(\frac12, 1)$ and $\alpha\geq\frac12$,  the family $\{\tilde Z_{N, \beta_N}^{\omega,h_N}\}_{N}$ of the rescaled partition functions is bounded in $L^2$-norm (at least for small $\hat{\beta}$ and $|\hat{h}|$) and in particular, for the limiting distribution we have $\E[\tilde Z_{\hat \beta, \hat h}]=1$.

Considering the above discussion together with Remarks~\ref{rem:highlight} and~\ref{remark:alpha}, we conjecture that the family $\{\tilde Z_N\}_{N}$ converges weakly to some $L^1$-integrable random variable under \ref{H}. In the rest of this section, we shall  prove the uniform integrability of $\{\tilde Z_N\}_{N}$ and characterize its weak limit, under a stronger condition:
\vspace{-0.2cm}
\begin{assumptionp}{(A2)}\label{H2}
Assume  $H\in(\frac12, 1)$ and $\alpha+2H>2$. 
\end{assumptionp}
\vspace{-0.2cm}
Condition \ref{H2} implies condition \ref{H} since $H<1$. Note that \ref{H2} is also the optimal condition that ensures the convergence of each chaos in the Stratonovich case (see Section \ref{sec:Str}).

\begin{lemma}\label{sup}
Assuming \ref{H2}, we have
\vspace{-0.2cm}
\begin{align}\label{e:sup-h}
\sup_{N}~\beta_N^2\E \Big[\sum_{n,m=1}^N\gamma(n-m) \mathbf1_{\{n\in \tau\}}\mathbf1_{\{m\in\tau\}}\Big]<\infty\quad\text{and}\quad\sup_{N}~\E\Big[\exp\Big\{h_N\sum_{n=1}^N \mathbf1_{\{n\in \tau\}}\Big\}\Big]<\infty.
\end{align}

\vspace{-0.5cm}
\end{lemma}

\begin{proof}
We prove the result for the case $\alpha<1$ and the other cases can be proven in a similar way. For the first bound in \eqref{e:sup-h}, we have
\vspace{-0.2cm}
 \begin{equation*}
    \begin{aligned}
        \quad\beta_N^2  \E \Big[\sum_{n,m=1}^N\gamma(n-m) \mathbf1_{\{n\in \tau\}}\mathbf1_{\{m\in\tau\}}\Big]&=\hat{\beta }^2\sum_{n,m=1}^N N^{2-2H}\gamma(n-m) \frac{L(N)^2}{N^{2\al-2}}\bP_\tau(n\in\tau,m\in\tau)\frac{1}{N^2}\\&\leq C\int_0^1\int_0^1 |s-t|^{2H-2}\phi_2^{\frac{1-\al+\epsilon}{1-\al}}(s,t)\dd s\dd t<\infty
    \end{aligned}
    \end{equation*}

\vspace{-0.15cm}
\noindent for $\alpha+2H>2$, where  $\phi_2$ is defined in (\ref{e:phi}) and the inequality follows from (\ref{con-gamma}), (\ref{con-gamma'}) and (\ref{ine:g-phi}). For the second bound in \eqref{e:sup-h}, we have that for some $\alpha'\in(0,\alpha)$
\vspace{-0.1cm}
\begin{equation}\label{e:exp-h}
\begin{split}
\mathbb{E}\Big[\exp\Big\{\sum_{n=1}^N h_N\mathbf1_{\{n\in \tau\}}\Big\}\Big]&=\sum\limits_{k=0}^\infty\frac{1}{k!}(h_N)^k\sum\limits_{\boldsymbol{n}\in\llbracket N\rrbracket^k}\bP_\tau(n_1,\cdots,n_k\in\tau)\\&=\sum\limits_{k=0}^\infty\frac{\hat{h}^k}{k!}\int_{[0,1]^k}g_{0,k}^{(N)}(t_{1},\cdots,t_{k})\prod\limits_{i=1}^k\dd t_{i}\leq\sum\limits_{k=0}^\infty\frac{C^k}{\Gamma(k\alpha')}<\infty.
\end{split}
\end{equation}

\vspace{-0.15cm}
\noindent where we have used \eqref{def:g_mk}, (\ref{ine:g-phi}), \eqref{e:phi} and Lemma~\ref{lem:a direct calculate}.
\end{proof}

\begin{proposition}\label{prop:UI}
Assuming \ref{H2}, the family $\big\{\tilde Z_N=\tilde Z_{N, \beta_N}^{\omega,h_N}\big\}_{N\in\bN}$ of the rescaled partition functions is uniformly integrable.
\end{proposition}
\begin{proof}
Recall \eqref{e:partition'} and we have
\vspace{-0.2cm}
\begin{equation*}
\bE_\omega\Big[\Big(\tilde{Z}_N\Big)^2\Big]=\bE_{\tau,\tau'}\Big[\exp\Big\{h_N\sum_{n=1}^N(1_{\{n\in \tau\}}+\mathbf1_{\{n\in\tau'\}})\Big\}\exp\Big\{ \beta_N^2\sum_{n,m=1}^N\gamma(n-m) \mathbf1_{\{n\in \tau\}}\mathbf1_{\{m\in\tau'\}}\Big\}\Big].
\end{equation*}

\vspace{-0.15cm}
\noindent Note that
\vspace{-0.2cm}
\begin{equation*}
\begin{aligned}
&\beta_N^2\sum\limits_{n,m=1}^N\gamma(n-m)\mathbf1_{\{n\in \tau,m\in\tau'\}}=\beta_N^2\sum\limits_{n,m=1}^N\E[\omega_n\omega_m]\mathbf1_{\{n\in \tau\}}\mathbf1_{\{m\in\tau'\}}\\
\le& \beta_N^2 \bigg(\E_\omega\Big[\Big(\sum_{n=1}^N \omega_n \mathbf1_{\{n\in \tau\}}\Big)^2\Big]\E_\omega\Big[\Big(\sum_{n=1}^N \omega_n \mathbf1_{\{n\in \tau'\}}\Big)^2\Big]\bigg)^{\frac12}\\
=&\beta_N^2 \Big( \sum_{n,m=1}^N\gamma(n-m) \mathbf1_{\{n\in \tau\}}\mathbf1_{\{m\in\tau\}}\Big)^\frac12\Big( \sum_{n,m=1}^N\gamma(n-m) \mathbf1_{\{n\in \tau'\}}\mathbf1_{\{m\in\tau'\}}\Big)^\frac12.
\end{aligned}
\end{equation*}

\vspace{-0.15cm}
\noindent We have that by $ab\leq\frac12(a^2+b^2)$,
\vspace{-0.1cm}
\begin{equation*}
\begin{aligned}
\bE_\omega\Big[\Big(\tilde{Z}_N\Big)^2\Big]&\leq \bE_{\tau,\tau'}\bigg[\exp\Big\{h_N\sum_{n=1}^N 1_{\{n\in \tau\}}+\frac{\beta_N^2}{2}
\sum_{n,m=1}^N\gamma(n-m) \mathbf1_{\{n\in \tau\}}\mathbf1_{\{m\in\tau\}}\Big\}\\&\quad\times \exp\Big\{h_N\sum_{n=1}^N 1_{\{n\in \tau'\}}+\frac{\beta_N^2}{2}
\sum_{n,m=1}^N\gamma(n-m) \mathbf1_{\{n\in \tau'\}}\mathbf1_{\{m\in\tau'\}}\Big\}   \bigg].
\end{aligned}
\end{equation*}

\vspace{-0.15cm}
Denote
\vspace{-0.2cm}
\[ A_l^{N}:= \Big\{ h_N\sum_{n=1}^N 1_{\{n\in \tau\}}+\frac{\beta_N^2}{2}
\sum_{n,m=1}^N\gamma(n-m) \mathbf1_{\{n\in \tau\}}\mathbf1_{\{m\in\tau\}}   \le l \Big\}.\]

\vspace{-0.1cm}
\noindent Lemma \ref{sup} and Markov's inequality together yield:  
\vspace{-0.1cm}
\begin{equation}\label{e:p-0}
\lim_{l\to\infty}\sup_{N}\bP\left((A_l^{N})^c\right)=0.
\end{equation}

\vspace{-0.15cm}
Define
\vspace{-0.2cm}
$$\tilde Z_N(A_l^{N}):=\bE_{\tau}\Big[\exp\Big\{h_N\sum_{x\in\Omega_{N}}\delta_{x}\mathbf 1_{A_l^N}\Big\}\exp\Big\{\diamond\Big(\beta_N\sum_{x\in\Omega_{N}}\omega_{Nx}\delta_{x}\mathbf 1_{A_l^N}\Big) \Big\}\Big].$$

\vspace{-0.15cm}
\noindent Then, by Lemma \ref{lem:keylem}, we get $
|\tilde Z_N -\tilde Z_N(A_l^{N} )|\le \E_{\tau}[ \mathscr E_N \mathbf 1_{(A_l^{N})^c}]+ \bP((A_l^{N})^c)$, where
\vspace{-0.1cm}
\[\mathscr E_N:=\exp\Big\{h_N\sum_{x\in\Omega_{N}}\delta_{x}\Big\}\exp\Big\{\diamond\Big(\beta_N\sum_{x\in\Omega_{N}}\omega_{Nx}\delta_{x}\Big) \Big\}.\]

\vspace{-0.15cm}
Thus we have
\vspace{-0.1cm}
\begin{equation}\label{e:diff-0}
\begin{aligned}
\E_\omega \big|\tilde Z_N -&\tilde Z_N(A_l^N)\big|\le \bE\big[ \mathscr E_N \mathbf 1_{(A_l^{N})^c} \big]+ \bP((A_l^{N})^c)=\E_\tau\Big[\mathbf1_{(A_l^{N})^c} \exp\Big(h_N\sum_{x\in\Omega_{N}}\delta_{x}\Big)  \Big]+ \bP((A_l^{N})^c)\\
&\leq \bP((A_l^{N})^c)^{\frac{1}{2}}\E_\tau\Big(\Big[\exp\Big(2h_N\sum_{x\in\Omega_{N}}\delta_{x}\Big)  \Big]\Big)^{\frac{1}{2}}+ \bP((A_l^{N})^c)\leq C \bP((A_l^{N})^c)^{\frac{1}{2}}+\bP((A_l^{N})^c),
\end{aligned}
\end{equation}

\vspace{-0.15cm}
\noindent where the constant $C$  in the last inequality is independent of $N$ by (\ref{e:sup-h}).

Finally, noting \eqref{e:p-0}, \eqref{e:diff-0}, and the fact that $\E [\tilde{Z}_N(A_l^{N})^2]\le e^{2l}$ for all $N$,
we can apply  Lemma~\ref{lem:UI} to obtain the uniform integrability of $\{\tilde Z_N\}_{N\in \bN}$.
\end{proof}

In the proof of Proposition \ref{prop:UI}, we applied the Cauchy-Schwarz inequality to decouple $\tau$ and $\tau'$, which requires the stronger Assumption~\ref{H2}. It seems that this is a technical issue and  \ref{H} should be sufficient for the uniform integrability. 
A direct consequence of the uniform integrability is the following, which characterizes the possible weak limit of $\{\tilde{Z}_N\}_{N}$. 
\begin{proposition}\label{prop:weak_limit_L1} The sequence $\{\tilde{Z}_N\}_{N\in\bN}$ is tight. Assuming \ref{H2},   we have $\bE[\mathcal{Z}]=1$ for any weak limit $\mathcal Z$. \end{proposition}
\begin{proof}
The tightness follows from  $\tilde Z_N>0$ a.s.\ and $\E\big[\tilde Z_N\big]=1$, for all $N\in\mathbb N$.  Assuming \ref{H2}, the family $\{\tilde{Z}_N\}_{N\in\bN}$ is uniformly integrable due to Proposition \ref{prop:UI} and thus $\E[\mathcal Z]=1$. 
\end{proof}



\section{The Stratonovich case}\label{sec:Str}

In this section, we shall study the weak convergence for the original partition function defined in \eqref{e:partition} and prove Theorem \ref{thm2}.  Our approach is similar to that in Section \ref{subsec:weak_skorohod} but with more involved details. First, we show that the limit \eqref{e:Z2} in Theorem \ref{thm2} is well-defined in $L^2$. To do that, we need good enough upper bounds for the norms of each chaos and the trace defined in \eqref{e:trace}. Then we show that each term in the Taylor expansion of the partition function converges weakly to a corresponding chaos. Theorem \ref{thm2} then follows from a classical truncation procedure.

\subsection{\texorpdfstring{$L^2$}--convergence of the continuum partition function.}\label{subsec:41} Define the following continuum partition function 
\vspace{-0.1cm}
\begin{equation}\label{continuum}
\begin{aligned}
Z_{\hat{\beta },\hat{h }}&=\sum_{k=0}^{\infty}\frac{1}{k!}\int_{[0,1]^{k}}\phi_{k}(t_{1},\cdots,t_{k})\prod_{i=1}^{k}\left(\hat{\beta }W(\dd t_{i})+\hat{h }\dd t_{i}\right)\\
&=\sum_{k=0}^{\infty}\sum_{p+q=k}\frac{\hat{\beta }^{p}\hat{h }^{q} }{p!q!}\mathbb{I}_{p}\bigg(\int_{[0,1]^{q}}\phi_{k}(t_{1},\cdots,t_{k})dt_{p+1}\cdots dt_{k}\bigg)\\&=\sum_{k=0}^{\infty}\sum_{p+q=k}\frac{\hat{\beta }^{p}\hat{h }^{q}}{p!q!}\mathbb{I}_{p}(\psi_{p,q})=\sum_{m=0}^{\infty}\sum_{k=0}^{\infty}\frac{\hat{\beta }^{m}\hat{h }^{k}}{m!k!}\mathbb{I}_{m}(\psi_{m,k}). 
\end{aligned}
\end{equation}

\vspace{-0.15cm}
Recall the functions $\psi$  in \eqref{e:psi}, $\phi$  in \eqref{e:phi}, and the trace defined in \eqref{e:trace}. 
\begin{lemma}\label{lem:Tr^j}
Assume \ref{H2}. 
Then we have $\mathrm{Tr} ^j {\psi}_{m,k}\in|\cH|^{\otimes(m-2j)}$ for $j=0,1,...,\lfloor\frac{m}{2}\rfloor$. Moreover,
\vspace{-0.1cm}
\begin{equation}\label{e:compare}
\Vert\mathrm{Tr} ^j {\psi}_{m,k}\Vert^2_{|\cH|^{\otimes (m-2j)}}\le \Vert\mathrm{Tr} ^j {\phi}_{m+k}\Vert^2_{|\cH|^{\otimes (m+k-2j)}}.
\end{equation}
\end{lemma}
\begin{proof} We only prove the case $\alpha\in(0,1)$. When $k=0$, we have $\psi_{m,0}=\phi_m$ and thus
\vspace{-0.2cm}
\begin{equation*}
\begin{aligned}
\Vert \mathrm{Tr} ^j {\phi}_m&\Vert^2_{|\cH|^{\otimes (m-2j)}}=\int_{[0,1]^{2(m-2j)}}\prod_{i=1}^{m-2j}|t_i-t_i'|^{2H-2}\mathrm{d}t_{i}\mathrm{d}t_{i}'\bigg(\int_{[0,1]^{2j}}\phi_m(s_1,\cdots,s_{2j},t_1,\cdots,t_{m-2j})\\&\times\prod\limits_{i=1}^j\frac{\mathrm{d}s_{2i-1}\mathrm{d}s_{2i}}{|s_{2i}-s_{2i-1}|^{2-2H}}\int_{[0,1]^{2j}}\phi_m(s_1',\cdots,s_{2j}',t_1',\cdots,t_{m-2j}')\prod\limits_{i=1}^j\frac{\mathrm{d}s_{2i-1}'\mathrm{d}s_{2i}'}{|s_{2i}'-s_{2i-1}'|^{2-2H}}\bigg)
\end{aligned}
\end{equation*}

\vspace{-0.15cm}
First, we split $[0,1]^{2m}$ into $(2m)!$ regions such that in each region, the numbers $\{t_i,t_i',s_j,s_j': i=1,\dots, m-2j, j=1, \dots, 2j\}$ are listed in an increasing order.  To get an upper bound, we replace, for instance, each  $|t_i-t_i'|^{2H-2}$ by $|t_i^*-t_i^\star |^{2H-2}$  where $t^*_{i}:=\max\{t_{i}, t_{i}'\}$ and $t_{i}^\star$ is the largest number that is smaller than $t^*_{i}$ (it could be some $t_j$, $t'_j$, $s_j$ or $s'_j$), and we perform the same operation for all the other $|s_{2i}-s_{2i-1}|^{2H-2}$, $|s'_{2i}-s'_{2i-1}|^{2H-2}$ and the factors in $\phi_m$.  Then on each region, the above integral is smaller than an integral with $2m$ variables in Lemma~\ref{lem:a direct calculate} with $m$ exponents equal to $\al-1$ and $m$ exponents equal to $2H+\al-3$. Noting that $2H+\alpha-3>-1$ under \ref{H2},  we get by Lemma~\ref{lem:a direct calculate}
\vspace{-0.1cm}
    \begin{equation}\label{e:est-trj}
        \Vert \mathrm{Tr} ^j {\phi}_m\Vert^2_{|\cH|^{\otimes (m-2j)}}\leq(2m)!\frac{\Gamma(\al)^m\Gamma(\al+2H-2)^m}{\Gamma(m(2H+2\alpha-2)+1)}<\infty.
    \end{equation}

\vspace{-0.15cm}
When $k>0$, by denoting $\boldsymbol{s}:=(s_1,\cdots,s_{2j})$, $\boldsymbol{t}:=(t_1,\cdots,t_{m-2j})$, $\boldsymbol{x}:=(x_1,\cdots,x_k)$ and $\boldsymbol{s}', \boldsymbol{t}', \boldsymbol{x}'$ correspondingly, we have
\vspace{-0.25cm}
\begin{align*}
&\quad\Vert \mathrm{Tr} ^j {\psi}_{m,k}\Vert^2_{|\cH|^{\otimes (m-2j)}}=\int_{[0,1]^{2(m-2j)}}\prod_{i=1}^{m-2j}|t_i-t_i'|^{2H-2}\mathrm{d}t_{i}\mathrm{d}t_{i}'\bigg(\int_{[0,1]^{2j}}\psi_{m,k}(\boldsymbol{s},\boldsymbol{t})\\
&\qquad\qquad\qquad\qquad\qquad \times\int_{[0,1]^{2j}}\psi_{m,k}(\boldsymbol{s}',t_1',\boldsymbol{t}')\prod\limits_{i=1}^j\frac{\mathrm{d}s_{2i-1}'\mathrm{d}s_{2i}'}{|s_{2i}'-s_{2i-1}'|^{2-2H}}\bigg)\\
&=\int_{[0,1]^{2(m-2j)}}\prod_{i=1}^{m-2j}|t_i-t_i'|^{2H-2}\mathrm{d}t_{i}\mathrm{d}t_{i}' \bigg(\int_{[0,1]^{2j}}\int_{[0,1]^k}\phi_{m+k}(\boldsymbol{s},\boldsymbol{t},\boldsymbol{x})\mathrm{d}\boldsymbol{x}\prod\limits_{i=1}^j\frac{\mathrm{d}s_{2i-1}\mathrm{d}s_{2i}}{|s_{2i}-s_{2i-1}|^{2-2H}}\\
&\qquad\qquad\qquad\qquad\qquad\qquad\qquad \times \int_{[0,1]^{2j}}\int_{[0,1]^k}\phi_{m+k}(\boldsymbol{s}',\boldsymbol{t}',\boldsymbol{x}')\mathrm{d}\boldsymbol{x'} \prod\limits_{i=1}^j\frac{\mathrm{d}s_{2i-1}'\mathrm{d}s_{2i}'}{|s_{2i}'-s_{2i-1}'|^{2-2H}}\bigg)\\
& \leq \int_{[0,1]^{2(m+k-2j)}}\prod_{i=1}^{m-2j}|t_i-t_i'|^{2H-2}\mathrm{d}t_{i}\mathrm{d}t_{i}'\prod_{p=1}^{k}|x_p-x_p'|^{2H-2}\mathrm{d}x_p\mathrm{d}x_p'\bigg(\int_{[0,1]^{2j}}\phi_{m+k}(\boldsymbol{s},\boldsymbol{t},\boldsymbol{x})\\
&\quad \times\prod\limits_{i=1}^j\frac{\mathrm{d}s_{2i-1}\mathrm{d}s_{2i}}{|s_{2i}-s_{2i-1}|^{2-2H}}\int_{[0,1]^{2j}}\phi_{m+k}(\boldsymbol{s}',\boldsymbol{t}',\boldsymbol{x}')\prod\limits_{i=1}^j\frac{\mathrm{d}s_{2i-1}'\mathrm{d}s_{2i}'}{|s_{2i}'-s_{2i-1}'|^{2-2H}}\bigg)=\Vert\mathrm{Tr} ^j {\phi}_{m+k}\Vert^2_{|\cH|^{\otimes (m+k-2j)}}.
\end{align*}

\vspace{-0.15cm}
\noindent This proves \eqref{e:compare} and hence $\mathrm{Tr} ^j {\psi}_{m,k}\in|\cH|^{\otimes(m-2j)}$ from the first step. 
\end{proof}

\begin{remark}\label{psi-bdd}
Assuming  \ref{H2},  Lemma \ref{lem:Tr^j} yields  ${\interleave \psi_{m,k}\interleave_{m}<\infty}$ (see \eqref{def-sm}) and hence $\mathbb I_m(\psi_{m,k})$ is well defined due to Hu--Meyer formula \eqref{H-M}.  On the other hand, $\alpha+2H>2$ is the optimal condition for ${\interleave \psi_{m,k}\interleave_{m}<\infty}$. Indeed, for $k=0, m=2$ and $j=1$, we have
\vspace{-0.1cm}
$$\mathrm{Tr} ^1 \phi_{2}=2\int_{0<t_1<t_2<1}\frac{C_\al^2}{t_1^{1-\al}(t_2-t_1)^{1-\al}}(t_2-t_1)^{2H-2}dt_1dt_2,$$ 

\vspace{-0.15cm}
\noindent which is finite if and only if $\al+2H>2$.
\end{remark}

We have shown that $\mathbb I_m(\psi_{m,k})$ is well defined under \ref{H2}. In order to make sense of $Z_{\hat{\beta },\hat{h }}$,   we  impose the following condition to obtain the $L^2$-convergence  for the series on the right-hand side of \eqref{continuum}. 

\vspace{-0.2cm}
\begin{assumptionp}{(A3)}\label{H3}
Assume  $H\in(\frac12, 1)$ and $\alpha+H>\frac32$. 
\end{assumptionp}
\vspace{-0.2cm}

 Note that the Assumption \ref{H3} implies the Assumption \ref{H2} due to $H>\frac12$.

\begin{proposition}\label{c-converge'}
Assume \ref{H3}. The continuum partition function $Z_{\hat{\beta },\hat{h }}$ given in $\mathrm{(\ref{continuum}) }$  is an $L^2$ random variable, that is, the series in \eqref{continuum} converges in $L^2$.
\end{proposition}
\begin{proof}
  Recalling  $Z_{\hat{\beta },\hat{h }}$ from \eqref{continuum}, we have
\vspace{-0.2cm}
\begin{equation*}
    \begin{aligned}
\|Z_{\hat \beta, \hat h}\|_2&\le \sum_{m=0}^{\infty}\sum_{k=0}^{\infty}\frac{\hat{\beta }^{m}|\hat{h }|^{k}}{m!k!}\|\mathbb{I}_{m}(\psi_{m,k})\|_2\leq \sum_{m=0}^{\infty}\sum_{k=0}^{\infty}\frac{\hat{\beta }^{m}|\hat{h }|^{k}}{m!k!}\sqrt{2^m (m+1)!\mathrm{Tr} ^m{\psi}_{2m,2k}}\\&\leq \sum_{m=0}^{\infty}\sum_{k=0}^{\infty}\frac{C^{m+k}}{\sqrt{m!}k!}\sqrt{\frac{[2(m+k)]!}{\Gamma((m+k)(2H+2\alpha-2)+1)}}\, ,
    \end{aligned}
\end{equation*}

\vspace{-0.15cm}
\noindent where the second inequality follows from Lemma \ref{Tr} below and the last inequality is due to \eqref{e:compare}  and \eqref{e:est-trj} above.

Note that $\alpha+H>\frac32$ implies $\zeta=4-2H-2\alpha<1$. By Stirling's formula, we get
\vspace{-0.1cm}
\begin{equation*}
\frac{C^{m+k}}{\sqrt{m!}k!}\sqrt{\frac{[2(m+k)]!}{\Gamma((m+k)(2H+2\alpha-2))}}\leq\sqrt{\frac{C^{m+k}(m+k)^{\zeta(m+k)}}{m^m k^{2k}}}\leq \sqrt{C^{m+k}\left(\frac{m^{\zeta(m+k)}}{m^m k^{2k}}+\frac{k^{\zeta(m+k)}}{m^m k^{2k}}\right)}.
\end{equation*}

\vspace{-0.15cm}
\noindent This together with Lemma \ref{lem:series-mk} yields $\|Z_{\hat \beta, \hat h}\|_2<\infty$. 
\end{proof}

\begin{lemma}\label{Tr}
Recall $\phi$ in \eqref{e:phi} and $\psi$ in \eqref{e:psi}. We have
\vspace{-0.1cm}
   \begin{equation*}
       \bE[\bI_m(\phi_m)^2]\leq 2^m (m+1)!\mathrm{Tr} ^m{\phi}_{2m}\quad\text{and}\quad\bE[\bI_m(\psi_{m,k})^2]\leq 2^m (m+1)!\mathrm{Tr} ^m{\psi}_{2m,2k}.
   \end{equation*}

\vspace{-0.25cm}
   \end{lemma}
   
   \begin{proof} For the first inequality,   Hu--Meyer formula (\ref{H-M}) implies
\vspace{-0.1cm}
   \begin{equation}\label{1}
   \begin{aligned}
       \mathbb{E}[|\mathbb{I}_m(\phi_m)|^2]&=\sum_{k=0}^{[\frac{m}{2}]}\frac{1}{(m-2k)!}\left(\frac{m!}{k!2^k}\right)^2 \Vert \mathrm{Tr} ^k \phi_m\Vert^2_{|\cH|^{\otimes (m-2k)}}=\sum_{k=0}^{[\frac{m}{2}]}\frac{m!}{2^{2k}}\binom{m}{2k}\binom{2k}{k} \Vert \mathrm{Tr} ^k \phi_m\Vert^2_{|\cH|^{\otimes (m-2k)}}\\&\leq\sum_{k=0}^{[\frac{m}{2}]}\frac{m!}{2^{2k}} 2^{m+2k}\Vert \mathrm{Tr} ^k \phi_m\Vert^2_{|\cH|^{\otimes (m-2k)}}= 2^m m!\sum_{k=0}^{[\frac{m}{2}]}\Vert \mathrm{Tr} ^k \phi_m\Vert^2_{|\cH|^{\otimes (m-2k)}},
   \end{aligned}
   \end{equation}

\vspace{-0.15cm}
\noindent where the inequality follows from the fact $\binom{m}{n}\le 2^m$.  For each term on the right-hand side, we have
\vspace{-0.2cm}
 \begin{equation}\label{3}
        \begin{aligned}
            \Vert \mathrm{Tr} ^k {\phi}_m\Vert^2_{|\cH|^{\otimes (m-2k)}}
=&\int_{[0,1]^{2m}}\prod_{i=1}^{m-2k}|t_i-t_i'|^{2H-2}\prod_{j=1}^k|s_{2j}-s_{2j-1}|^{2H-2}|s_{2j}'-s_{2j-1}'|^{2H-2}\\&\times\quad\phi_m(s_1,\cdots,s_{2k},t_1,\cdots,t_{m-2k})\phi_m(s_1',\cdots,s_{2k}',t_1',\cdots,t_{m-2k}')\dd\boldsymbol{s}\dd\boldsymbol{s'}\dd\boldsymbol{t}\dd\boldsymbol{t'}\\ \le& \mathrm{Tr}^m {\phi}_{2m}.
        \end{aligned}
    \end{equation}
 For the last inequality above,  denoting $\boldsymbol{t}:=(t_1, \cdots, t_{m-2k})$, $\boldsymbol{s}:=(s_1,\cdots,s_{2k})$ and $\boldsymbol{t}', \boldsymbol{s}'$ correspondingly, we have that by the symmetry and the expression of $\phi$,
\vspace{-0.2cm}
\begin{equation*}
\phi_m(\boldsymbol{s},\boldsymbol{t})\phi_m(\boldsymbol{s}',\boldsymbol{t}')\leq\phi_{2m}(\boldsymbol{s},\boldsymbol{s}',t_1,t'_1,\cdots,t_{m-2k},t'_{m-2k}),
\end{equation*}

\vspace{-0.15cm}
\noindent since each subinterval among $(\boldsymbol{s},\boldsymbol{t})$ and $(\boldsymbol{s}',\boldsymbol{t}')$ is shortened by the interlacement between $(\boldsymbol{s},\boldsymbol{t})$ and $(\boldsymbol{s}',\boldsymbol{t}')$. Also note that the kernel $K(\cdot)$ in the trace \eqref{e:trace} is not influenced by the permutation in $\phi$. Then substituting \eqref{3}  into (\ref{1}) yields the desired inequality.

Similar to the above, noting that
\vspace{-0.1cm}
    \begin{equation}\nonumber
    \begin{aligned}
       \psi_{m,k}(\boldsymbol{s},\boldsymbol{t})\psi_{m,k}(\boldsymbol{s}',\boldsymbol{t}')\leq\psi_{2m,2k}(\boldsymbol{s},\boldsymbol{s}', t_1,t_1',\cdots,t_{m-2k},t_{m-2k}'),
    \end{aligned}
\end{equation}

\vspace{-0.15cm}
\noindent we can prove the second inequality in the same way. 
\end{proof}

\subsection{Weak convergence of the rescaled partition functions.} First, note that Taylor expansion of the partition function $Z_N:=Z_{N,\beta_N}^{\omega,h_N}$ given in \eqref{e:partition} yields
\vspace{-0.1cm}
\begin{equation}\label{discrete}
\begin{aligned}
Z_N&=\mathbb{E}_{\tau}\bigg[\Big(\sum_{k=0}^{\infty}\frac{1}{k!}\Big(\sum_{n=1}^Nh_N \delta_{\frac{n}{N}}\Big)^k\Big)\Big(\sum_{m=0}^{\infty}\frac{1}{m!}\Big(\sum_{n=1}^N\beta_N\omega_n \delta_{\frac{n}{N}}\Big)^m\Big)\bigg]
\\&=\sum_{m=0}^{\infty}\sum_{k=0}^{\infty}\frac{1}{k!m!}\sum_{ \boldsymbol{n}\in\llbracket N \rrbracket^{m+k}}\beta_N^m h_N^k\mathbb{P}_{\tau}(n_1,\cdots,n_m,\cdots,n_{m+k}\in\tau)\prod_{i=1}^m\omega_{n_i}=:\sum_{m=0}^{\infty}\sum_{k=0}^{\infty}\frac{1}{k!m!}\mathbb{S}_{m,k}^{(N)}.
\end{aligned}
\end{equation}

\vspace{-0.15cm}
We then prove the weak convergence of each chaos.
\begin{proposition}\label{prop-S-I}
 Assume \ref{H2}. Then, for each $m,k\in\bN_0$,
 \vspace{-0.1cm}
\begin{equation}
\mathbb{S}_{m,k}^{(N)}\overset{(d)}{\longrightarrow}\hat{\beta }^{m}\hat{h }^{k}\mathbb{I}_{m}(\psi_{m,k}),\quad\text{as}~N\rightarrow\infty,
\label{prop-4}
\end{equation}

\vspace{-0.15cm}
\noindent where $\psi_{m,k}$ is defined in~(\ref{e:psi}) and $ \mathbb{I}_{m}(\psi_{m,k})$ is an $m$-th multiple Stratonovich integral. Moreover, for any $p\in\bN$\ and $l_{1},\cdots,l_{p},k_{1},\cdots,k_{p}\in\bN_0$, we have the joint convergence in distribution
\vspace{-0.1cm}
\begin{equation}
\left( \mathbb{S}_{l_1,k_{1}}^{(N)},\cdots, \mathbb{S}_{l_p,k_{p}}^{(N)})\right)\overset{(d)}{\longrightarrow}\left(\hat{\beta}^{l_1} \hat{h }^{k_{1}}\mathbb{I}_{l_1}(\psi_{l_{1},k_{1}}),\cdots,\hat{\beta}^{l_p} \hat{h }^{k_{p}}\mathbb{I}_{l_p}(\psi_{l_{p},k_{p}})\right),\quad\text{as}~N\rightarrow\infty.
\label{prop-4-multi}
\end{equation}

\vspace{-0.3cm}
\end{proposition}

\begin{proof}
We first prove $\mathrm{(\ref{prop-4}) }$. Recalling the definition of $\mathbb{S}_{m,k}^{(N)}$ in \eqref{discrete},  we get   by Lemma~\ref{ord_prod_wick_prod},\begin{equation*}
\begin{aligned}
\mathbb{S}_{m,k}^{(N)}&=\sum_{ \boldsymbol{n}\in\llbracket N \rrbracket^{m+k}}\beta_N^m h_N^k\mathbb{P}_{\tau}(n_1,\cdots,n_{m+k}\in\tau)\sum_{j=0}^{\lfloor\frac{m}{2}\rfloor}\sum_{B\subset\llbracket m \rrbracket,~|B|=m-2j}\omega_{\boldsymbol{n}_B}^\diamond\mathbb{E}[\omega_{\boldsymbol{n}_{\llbracket m \rrbracket\backslash B}}]
\\&=\sum_{j=0}^{\lfloor\frac{m}{2}\rfloor}\sum_{B\subset\llbracket m \rrbracket,~|B|=m-2j}\beta_N^m h_N^k\sum_{ \boldsymbol{n}\in\llbracket N \rrbracket^{m+k}}\mathbb{P}_{\tau}(n_1,\cdots,n_{m+k}\in\tau)\omega_{\boldsymbol{n}_B}^\diamond\mathbb{E}[\omega_{\boldsymbol{n}_{\llbracket m \rrbracket\backslash B}}].
\end{aligned}
\end{equation*}

For each $j=0, 1,\cdots,\lfloor\frac{m}{2}\rfloor$, all the terms in the summation $\sum_{\substack{B\subset\llbracket m \rrbracket, |B|=m-2j}}$ are equal by symmetry. Fix $j\in\{0,1,...,\lfloor\frac{m}{2}\rfloor\}$ and note that there are in total $\binom{m}{m-2j}$ subsets of $\llbracket m \rrbracket$ with cardinality $m-2j$. Without loss of generality, we may assume $B=\{2j+1,...,m\}$. By the Wick theorem (see, e.g., \cite[Theorem 1.28]{J97}), we have
\vspace{-0.1cm}
\[\mathbb{E}[\omega_{\boldsymbol{n}_{\llbracket m \rrbracket\backslash B}}]=\mathbb{E}[\omega_{\boldsymbol{n}_{\llbracket 2j \rrbracket}}]=\sum_V\prod_{\{l_1,l_2\}\in V}\mathbb{E}[\omega_{n_{l_1}}\omega_{n_{l_2}}],\]

\vspace{-0.15cm}
\noindent where the sum $\sum_V$ is taken over all pair partitions $V$ of $\llbracket 2j \rrbracket$. By symmetry, the summations
\vspace{-0.1cm}
\[\sum_{ \boldsymbol{n}\in\llbracket N \rrbracket^{m+k}}\omega_{\boldsymbol{n}_B}^\diamond\prod_{\{l_1,l_2\}\in V}\mathbb{E}[\omega_{n_{l_1}}\omega_{n_{l_2}}]\mathbb{P}_{\tau}(n_1,\cdots,n_{m+k}\in\tau)\]

\vspace{-0.15cm}
\noindent equal to each other for all pair partitions $V$. Since there are in total $(2j-1)!!$ pair partitions of $\llbracket 2j \rrbracket$,
\vspace{-0.1cm}
\begin{equation}
\begin{aligned}
\mathbb{S}_{m,k}^{(N)}&=\sum_{j=0}^{\lfloor\frac{m}{2}\rfloor} \binom{m}{m-2j} (2j-1)!!\beta_N^m h_N^k\sum_{ \boldsymbol{n}\in\llbracket N \rrbracket^{m+k}}\omega_{\boldsymbol{n}_B}^\diamond\prod_{l=1}^{j}\mathbb{E}[\omega_{n_{2l-1}}\omega_{n_{2l}}]\mathbb{P}_{\tau}(n_1,\cdots,n_{m+k}\in\tau)
\\&=\sum_{j=0}^{[\frac{m}{2}]}\frac{m!}{j!(m-2j)!2^j}\beta_N^m h_N^k\sum_{ \boldsymbol{n}\in\llbracket N \rrbracket^{m+k}}\omega_{\boldsymbol{n}_B}^\diamond\prod_{l=1}^{j}\mathbb{E}[\omega_{n_{2l-1}}\omega_{n_{2l}}]\mathbb{P}_{\tau}(n_1,\cdots,n_{m+k}\in\tau).
\label{Smk}
\end{aligned}
\end{equation}

\vspace{-0.15cm}
Compare (\ref{Smk}) with Hu--Meyer formula (\ref{H-M}) and note that $\mathrm{Tr}^{j}\psi_{m,k}\in|\cH|^{\otimes(m-2j)}$ for $k=0,1,...,\lfloor\frac{m}{2}\rfloor$ by Lemma $\mathrm{ \ref{lem:Tr^j}}$. From now on, without loss of generality, we assume that $m$ is odd and the case of even $m$ can be treated in the same manner. By Proposition \ref{prop:2-joint}, we have
\vspace{-0.1cm}
\begin{equation}
\begin{aligned}
\left(\I_{1}^{(N)}(\mathrm{Tr}^{\frac{m-1}{2}}\psi_{m,k}),\cdots,\I_{m}^{(N)}(\psi_{m,k})\right)
\overset{(d)}{\longrightarrow}\left(\I_{1}(\mathrm{Tr}^{\frac{m-1}{2}}\psi_{m,k}),\cdots,\I_{m}(\psi_{m,k})\right),\quad\text{as}~N\rightarrow\infty.
\label{prop-2-multi}
\end{aligned}
\end{equation}

\vspace{-0.15cm}
\noindent Thus, in order to prove (\ref{prop-4}), it suffices to show that for all $j=0,1,\cdots,\frac{m-1}{2}$, 
\vspace{-0.1cm}
\begin{equation}\label{e:Y20}
\begin{aligned}
Y_N:&=\frac{L(N)^{m+k}}{N^{m(\alpha-1+H)+\alpha k}}\sum_{\boldsymbol{n}\in \llbracket N \rrbracket^{m+k}}\omega_{\boldsymbol{n}_B}^\diamond\prod_{l=1}^{j}\mathbb{E}[\omega_{n_{2l-1}}\omega_{n_{2l}}]\mathbb{P}_{\tau}(n_1,\cdots,n_{m+k}\in\tau)-\I_{m-2j}^{(N)}(\mathrm{Tr}^j\psi_{m,k})\\
&=\I_{m-2j}^{(N)}(f_{\mathrm{Tr}^j,m,k}^{(N)})-\I_{m-2j}^{(N)}(\mathrm{Tr}^j\psi_{m,k})\overset{L^2}{\longrightarrow}0\quad\text{as}~N\to \infty,
\end{aligned}
\end{equation}
\vspace{-0.15cm}
\noindent with
\vspace{-0.1cm}
\begin{equation}
\begin{aligned}
&f_{\mathrm{Tr}^j,m,k}^{(N)}=f_{\mathrm{Tr}^j,m,k}^{(N)}(t_{1},\cdots,t_{m-2j})
\\
&:=N^{-2jH}N^{-k}\frac{L(N)^{m+k}}{(N^{\alpha-1})^{m+k}}\!\!\sum_{\boldsymbol{u}\in\llbracket N \rrbracket^{2j}}\sum_{\boldsymbol{q}\in\llbracket N \rrbracket^{m-2j}}\sum_{\boldsymbol{s}\in\llbracket N \rrbracket^{k}}\prod_{l=1}^{j}\mathbb{E}[\omega_{n_{2l-1}}\omega_{n_{2l}}]P_{\tau}(\boldsymbol{u},\boldsymbol{q},\boldsymbol{s})\prod_{i=1}^{m-2j}\mathbf1_{(\frac{q_{i}-1}{N},\frac{q_{i}}{N}]}(t_{i}),
\end{aligned}
\end{equation}

\vspace{-0.15cm}
\noindent where we recall that $\mathbf I_{m-2j}^{(N)}$ is defined in \eqref{e:I(f)},  $P_{\tau}(\boldsymbol{u},\boldsymbol{q},\boldsymbol{s})$ is defined in \eqref{def:P_tau}. 

For $j=0$, we have $Y_N=\I_{m}^{(N)}(f_{m,k}^{(N)}-\psi_{m,k})$, which converges to $0$ in $L^2$ as $N\to\infty$ as has been proved in the proof of  Proposition \ref{prop:weak-con-Sk}. For $j=1,...,\frac{m-1}{2}$, by 
Lemma \ref{lem:isometry} and Lemma \ref{lem:1/H norm}, we get 
\vspace{-0.1cm}
\begin{equation}
\begin{aligned}
\mathbb{E}[Y_N^2]&\lesssim (m-2j)!\|f_{\mathrm{Tr}^j,m,k}^{(N)}-\mathrm{Tr}^j\psi_{m,k}\|_{|\cH|^{\otimes(m-2j)}}^2\\
&\lesssim(m-2j)!\int_{[0,1]^{2(m-2j)}}\prod_{i=1}^{m-2j}|t_i-t'_i|^{2H-2}|\boldsymbol{D}_N \boldsymbol{D}'_N|\dd\boldsymbol{t}\dd\boldsymbol{t}'\lesssim (m-2j)!\|\boldsymbol{D}_N\|_{\frac{1}{H}}^2
\label{Y-bdd}
\end{aligned}
\end{equation}

\vspace{-0.15cm}
\noindent where
\vspace{-0.1cm}
\begin{equation}
\begin{aligned}
\boldsymbol{D}_N=f_{\mathrm{Tr}^j,m,k}^{(N)}-\mathrm{Tr}^j\psi_{m,k}\quad\text{and}\quad{\boldsymbol{D}}'_N=\boldsymbol{D}_N(\boldsymbol{t}').
\end{aligned}
\end{equation}

\vspace{-0.15cm}
Then, in order to prove \eqref{e:Y20}, it suffices to show 
\vspace{-0.1cm}
\begin{equation}\label{e:f2p}
f_{\mathrm{Tr}^j,m,k}^{(N)}\xrightarrow{L^{1/H}}\mathrm{Tr}^j\psi_{m,k}\quad\text{ as } N\to\infty.
\end{equation}

\vspace{-0.15cm}
\noindent We denote
\begin{align*}
&g_{\mathrm{Tr}^j,m,k}^{(N)}(t_1,\cdots,t_{m+k})\\
&:=\frac{L(N)^{m+k}}{(N^{\alpha-1})^{m+k}}\sum_{\boldsymbol{n}\in\llbracket N \rrbracket^{m+k}}\left(\prod_{l=1}^{j}\gamma_N\left(\frac{n_{2l-1}}{N}-\frac{n_{2l}}{N}\right)\bP_{\tau}\left(n_1,\cdots ,n_{m+k}\in\tau\right)\prod_{i=1}^{m+k}\I_{(\frac{n_{i}-1}{N},\frac{n_{i}}{N}]}(t_{i})\right),
\end{align*}

\vspace{-0.15cm}
where $\gamma_{N}(t)$ is defined in \eqref{def:gamma_N}, and denote
\vspace{-0.1cm}
\begin{align*}
&\quad\quad\phi_{\mathrm{Tr}^j,m+k}(t_1,\cdots,t_{m+k}):=\prod_{l=1}^{j}|t_{2j-1}-t_{2j}|^{2H-2}\phi_{m+k}(t_1,\cdots,t_{m+k}).
\end{align*}

\vspace{-0.15cm}
\noindent Then we have
\vspace{-0.1cm}
\begin{equation*}
\begin{aligned}
f_{\mathrm{Tr}^j,m,k}^{(N)}(t_{1},\cdots,t_{m-2j})&=\int_{[0,1]^{2j+k}}g_{\mathrm{Tr}^j,m,k}^{(N)}(t_1,\cdots,t_{m+k})\dd t_1\cdots\dd t_{2j}\dd t_{m+1}\cdots\dd t_{m+k}, \\
\mathrm{Tr}^j\psi_{m,k}(t_{1},\cdots,t_{m-2j})&=\int_{[0,1]^{2j+k}}\phi_{\mathrm{Tr}^j,m+k}(t_1,\cdots,t_{m+k})\dd t_1\cdots\dd t_{2j}\dd t_{m+1}\cdots\dd t_{m+k}.
\end{aligned}
\end{equation*}

\vspace{-0.15cm}
To prove \eqref{e:f2p}, by Jensen's inequality, it suffices to show $g_{\mathrm{Tr}^j,m,k}^{(N)}\xrightarrow{L^{1/H}}\phi_{\mathrm{Tr}^j,m+k}$, as $N\to\infty$. First, by (\ref{renewl theorem}) and (\ref{con-gamma'}), we have $g_{\mathrm{Tr}^j,m,k}^{(N)}\to \phi_{\mathrm{Tr}^j,m+k}$ pointwisely as $N\to\infty$. By Lemma ~\ref{lem:Tr^j}, $\phi_{\mathrm{Tr}^j,m+k}\in L^{1/H}$. Hence, it suffices to exhibit an $L^{1/H}$ domination of $\big\{g_{\mathrm{Tr}^j,m,k}^{(N)}\big\}_{N\geq 1}$. By \eqref{local} and \eqref{Karamata}, and
together with \eqref{con-gamma} and \eqref{con-gamma'}, we have that for some $\alpha'\in(0,\alpha)$,
\vspace{-0.1cm}
\begin{equation}\label{1/H-domination}
 g_{\mathrm{Tr}^j,m,k}^{(N)}(t_1,\cdots,t_{m+k})\leq C^{m+k}\prod_{l=1}^{j}|t_{2l-1}-t_{2l}|^{2H-2}\big(\phi_{m+k}(t_1,\cdots,t_{m+k})\big)^{\frac{1-\al'}{1-\al}}.
\end{equation}

\vspace{-0.15cm}
\noindent By a similar argument used in  the proof of  Lemma \ref{lem:Tr^j}, we can show that the right-hand side of (\ref{1/H-domination}) belongs to $L^{1/H}$. This provides the desired $L^{1/H}$ domination.

Finally, the proof is completed by the continuous mapping theorem  and \eqref{prop-2-multi}.
\end{proof}

\subsection{\texorpdfstring{$L^1$}--bound of the rescaled partition functions}\label{Lp_bound}

In this subsection, we prove the uniform $L^1$-boundedness, under Assumption \ref{H3},  for the partition function $Z_N$. Recall its Taylor expansion \eqref{discrete} and by the  Cauchy-Schwarz inequality, we have that
\vspace{-0.1cm}
\begin{equation}\label{eq:L1_ZN}
\begin{split}
\bE[Z_N]\leq\sum\limits_{k=0}^\infty\sum\limits_{m=0}^\infty\frac{1}{k!}\frac{1}{m!}\bE_\tau\Big[\Big(\sum\limits_{n=1}^N h_N \delta_{\frac{n}{N}}\Big)^{2k}\Big]^{\frac12}\bE\Big[\Big(\sum\limits_{n=1}^N\beta_N\omega_n\delta_{\frac{n}{N}}\Big)^{2m}\Big]^{\frac12}.
\end{split}
\end{equation}

\vspace{-0.15cm}
Using the same argument as in \eqref{e:exp-h}, we get that for some $\alpha'\in(0,\alpha)$,
\vspace{-0.1cm}
\begin{equation}\label{eq:h}
\sum\limits_{k=0}^\infty\frac{1}{k!}\bE_\tau\Big[\Big( \sum\limits_{n=1}^N h_N\delta_{\frac{n}{N}}\Big)^{2k}\Big]^{\frac12}\leq\sum\limits_{k=0}^\infty\frac{C^k}{k!}\sqrt{\frac{(2k)!}{\Gamma(2k\alpha')}}\lesssim\sum\limits_{k=0}^\infty\frac{C^k}{k^{k\alpha'}}<\infty.
\end{equation}
\vspace{-0.15cm}

For each term in the summation over $m$ in \eqref{eq:L1_ZN}, we have that
\vspace{-0.1cm}
\begin{equation}\label{eq:chaos_w}
\begin{split}
\bE\Big[\Big(\sum\limits_{n=1}^N\beta_N\omega_n\delta_{\frac{n}{N}}\Big)^{2m}\Big]^{\frac12}=\Big(\sum\limits_{\boldsymbol{n}\in\llbracket N\rrbracket^{2m}}\beta_N^{2m}\bP_\tau(n_1,\cdots,n_{2m}\in\tau)\bE_\omega[\omega_{\boldsymbol{n}_{\llbracket m\rrbracket}}]\Big)^{\frac12}.
\end{split}
\end{equation}

\vspace{-0.15cm}
\noindent Similar to the proof of Proposition \ref{prop-S-I}, $\bE[\omega_{\boldsymbol{n}_{\llbracket m\rrbracket}}]=\sum_V\prod_{\{l_1,l_2\}\in V}\bE[\omega_{l_1}\omega_{l_2}]$, where $V$ is all pair partitions of $\llbracket 2m\rrbracket$. Note that $|V|=(2m-1)!!$. By \eqref{local} and  \eqref{Karamata}, a Riemann sum approximation and the kernel modification procedure in the proof of Lemma \ref{lem:Tr^j}, \eqref{eq:chaos_w} is bounded above by
\vspace{-0.1cm}
\begin{equation}
\sqrt{\frac{C^{2m}(2m)!(2m-1)!!}{\Gamma(m(2\alpha'+2H-2))}}\lesssim C^m m^m\sqrt{\frac{m}{m^{m(2\alpha'-2H-3)}}}
\end{equation}

\vspace{-0.15cm}
\noindent for some $\alpha'<\alpha$ and $\alpha'+H>\frac32$. Hence, we have that
\vspace{-0.2cm}
\begin{equation}\label{eq:w}
\sum\limits_{m=0}^\infty\frac{1}{m!}\bE\Big[\Big(\sum\limits_{n=1}^N\beta_N\omega_n \delta_{\frac{n}{N}}\Big)^{2m}\Big]^{\frac12}\leq\sum\limits_{m=0}^\infty\frac{C^m}{m^{m(\alpha'+H-\frac32)}}<\infty. 
\end{equation}

\vspace{-0.15cm}
Combining \eqref{eq:L1_ZN}, \eqref{eq:h} and \eqref{eq:w} we conclude the $L^1$-boundedness of $Z_N$ and we can also deduce from the proof that
\vspace{-0.2cm}
\begin{equation}\label{sumM-0}
\begin{aligned}
\lim_{M\to\infty}\sup_{N\in\mathbb{N}}\sum_{l=M}^{\infty}\sum_{m+k=l}\frac{1}{k!m!}\mathbb{E}\Big[\Big|h_N\sum_{n=1}^N \delta_{\frac{n}{N}}\Big|^{k}\Big|\beta_N\sum_{n=1}^N\omega_{n} \delta_{\frac{n}{N}}\Big|^{m}\Big]=0.
\end{aligned}
\end{equation}

\vspace{-0.15cm}
Now we are ready to prove our second main result. 
\begin{proof}[Proof of Theorem \ref{thm2}]
By Proposition \ref{prop-S-I}, we have that for all $R\in \bN$,
\vspace{-0.1cm}
\begin{equation*}
Z^{(R)}_{N}=\sum_{l=0}^{R}\sum_{m+k=l}\frac{1}{k!m!}\mathbb{S}_{m,k}^{(N)}\overset{(d)}{\longrightarrow}Z^{(R)}:=\sum_{l=0}^{R}\sum_{m+k=l}\frac{\hat{\beta }^{m}\hat{h }^{k}}{m!k!}\mathbb{I}_{m}(\psi_{m,k}),\quad\text{as}~N\to\infty.
\end{equation*}

\vspace{-0.15cm}
By Proposition \ref{c-converge'}, $Z^{(R)}\overset{L^2}{\longrightarrow}Z$ as $N\to\infty$. By Lemma \ref{lem:weak-con}, we only need to show that $Z^{(R)}_{N}$ converges to $Z_N$ (see \eqref{discrete}) in probability uniformly in $N$ as $R\to\infty$. This follows from \eqref{sumM-0}.
\end{proof}

\begin{appendix}
\section{Some miscellaneous results}\label{app:A}

The following two lemmas about the Wick product can be found in \cite[Corollary 3.17, Theorem 3.9]{J97}. For Lemma \ref{ord_prod_wick_prod}, see also \cite[equation (A.5)]{rsw24}.
\vspace{-0.1cm}
\begin{lemma}\label{ord_prod_wick_prod}
Let $(\omega_n)_{n\in\bN}$ be Gaussian. Then
\[\omega_{\boldsymbol{n}_{\llbracket m\rrbracket}}=\sum_{j=0}^{[\frac{m}{2}]}\sum_{B\subset\llbracket m \rrbracket, ~|B|=m-2j}\omega_{\boldsymbol{n}_B}^\diamond\mathbb{E}[\omega_{\boldsymbol{n}_{\llbracket m \rrbracket\backslash B}}],\]
where $\omega_{\boldsymbol{n}_{\llbracket m\rrbracket}}:=\prod_{i=1}^m\omega_{n_i}$ is an ordinary product  and $\omega_{\boldsymbol{n}_B}^\diamond:=\prod_{n_i\in B}\diamond\omega_{n_i}$ is a Wick product. 
\end{lemma}

\begin{lemma}\label{lem:Wick}
Let $(\omega_n)_{n\in\bN}$ be Gaussian. Then, for $\boldsymbol{n}=(n_1, \dots, n_p)$ and $\boldsymbol{m}=(m_1, \dots, m_q)$,
\vspace{-0.1cm}
\begin{equation*}
\E[\omega_{\boldsymbol{n}}^{\diamond}\omega_{\boldsymbol{m}}^{\diamond} ] =\begin{cases}
    \displaystyle0 & \text{ if } p\neq q,\\[10pt]
    \displaystyle\sum_{\pi} \prod_{\{i,j\}\in \pi } \E[\omega_i\omega_j], & \text{ if } p=q,
\end{cases}
\end{equation*}

\vspace{-0.15cm}
\noindent where $\pi$ is a pairwise partition of the set $\{n_1,\dots, n_p, m_1, \dots, m_p\}$ such that for each $\{i, j\}\in \pi$, $i\in \{n_1,\dots, n_p\}$ and $j\in \{m_1, \dots, m_p\}$.
\end{lemma}

The following lemma about the commutative diagram of weak convergence can be found in \cite[Theorem 3.2, Chapter 1]{bill99}.
\begin{lemma}\label{lem:weak-con}
\vspace{-0.1cm}
Consider random variables $Y_n^{(N)}, Y^{(N)}, Y_n$ and $Y$ such that $Y^{(N)}_n$ converges weakly to $Y_n$ as $N\to \infty$, $Y_n$ converges weakly to $Y$ as $n\to\infty$, and $Y_n^{(N)}$ converges in probability  to $Y^{(N)}$ uniformly in $N$ as $n\to\infty$, then we have $Y^{(N)}$ converges  weakly to $Y$ as $N\to\infty$. 
\end{lemma}

The following inequality is borrowed from  \cite{mmv01}.
\vspace{-0.1cm}
\begin{lemma}\label{lem:1/H norm}
For $H\in(\frac{1}{2},1]$, there exists a constant $C_H$ depending only on $H$, such that,
\vspace{-0.1cm}
$$\int_{\bR^m}\int_{\bR^m}f(\boldsymbol{t})f(\boldsymbol{s})\prod_{i=1}^m|t_i-s_i|^{2H-2}\dd\boldsymbol{t}\dd\boldsymbol{s}\leq C_H^m \|f\|_{L^{1/H}(\bR^m)}^2, $$

\vspace{-0.5cm}    
\end{lemma}
The following result can be obtained by a direct computation. 
\begin{lemma}\label{lem:a direct calculate}
Suppose $\al_i>-1$ for $i=1,\cdots,m$. Denoting $\al:=\sum_{i=1}^m \al_i$, then
\vspace{-0.1cm}
$$\int_{0=r_0<r_1<r_2<\cdots<r_m<t}\prod_{i=1}^m(r_i-r_{i-1})^{\al_i}\dd\boldsymbol{r}=\frac{\prod_{i=1}^m\Gamma(1+\al_i)}{\Gamma(m+\al+1)}t^{m+\al},  $$

\vspace{-0.15cm}
\noindent where $\Gamma(x)=\int_0^\infty t^{x-1} e^{-t}\dd t$ is the Gamma function.
    
\end{lemma}

\begin{lemma}\label{lem:series-mk}
Let $a, b$  and $c$ be positive constants. Assuming $c>a$, then we have for all $A >0$,
\vspace{-0.1cm}
\[\sum_{k+m=1}^\infty  A^{m+k} \frac{m^{a k}}{m^{bm} k^{c k}}<\infty.\]

\vspace{-0.5cm}
\end{lemma}
\begin{proof}
For some large $k_0$ and $m_0$ which will be determined later, we write
\vspace{-0.1cm}
\begin{equation*}
\sum_{k+m=1}^\infty A^{m+k} \frac{m^{a k}}{m^{bm} k^{c k}}\leq\Big(\sum_{k<k_0}\sum_{m\geq 1}+\sum_{m<m_0}\sum_{k\geq 1}+\sum_{m\geq m_0}\sum_{k\geq k_0}\Big) A^{m+k} \frac{m^{a k}}{m^{bm} k^{c k}}.
\end{equation*}

\vspace{-0.15cm}
\noindent It is not hard to see that the first two sums are finite. For the third term, note that $(bm)^{ck}(ck)^{bm}\leq(bm)^{bm}(ck)^{ck}$. Let $d=c-a$ and we have that
\vspace{-0.1cm}
\begin{equation*}
\begin{aligned}
&\sum_{m\geq m_0}\sum_{k\geq k_0} A^{m+k} \frac{m^{a k}}{m^{bm} k^{c k}}\leq \sum_{m\geq m_0}\sum_{k\geq k_0} A^{m+k} \frac{m^{ak}}{m^{ck} k^{bm}}\Big(\frac{b}{c}\Big)^{bm}\Big(\frac{c}{b}\Big)^{ck}\\
=&\sum_{m\geq m_0}\sum_{k\geq k_0}\frac{1}{(C_A m)^{dk}(C'_A k)^{bm}}\leq\sum_{m\geq m_0}\frac{1}{(C'_A k_0)^{bm}}\sum_{k\geq k_0}\frac{1}{(C_A m_0)^{dk}}<\infty,
\end{aligned}
\end{equation*}

\vspace{-0.15cm}
\noindent by choosing $C_A'k_0>1$ and $C_A m_0>1$. The proof is now completed.   
\end{proof}


\begin{lemma}\label{lem:keylem} Let $X$ be an exponentially integrable random variable on $(\Omega, \mathcal F, \bP)$ and $A\in \mathcal F$. Then
\vspace{-0.1cm}
\[\bE\left[ e^X -e^{X\mathbf 1_A} \right]= \bP(
A^c) \Big(\bE\Big[e^{X}\Big| A^c\Big]-1\Big)=\E\left[e^X\mathbf 1_{A^c}\right]-\bP(A^c).\]

\vspace{-0.5cm}
\end{lemma}
 \begin{proof}
 In general, we have
 \vspace{-0.1cm}
 \begin{align*}
\bE\left[ f(X) -f(X\mathbf 1_A)\right]&=\bP(A) \E[f(X)-f(X\mathbf 1_A)|A]+ \bP(A^c) \E [f(X) -f(X\mathbf 1_A)|A^c]\\
& = \bP(A)\cdot 0+\bP(A^c) \E [f(X) -f(0)|A^c]= \E[f(X) \mathbf 1_{A^c}] -f(0)\bP(A^c),
\end{align*}

\vspace{-0.15cm}
\noindent which yields the desired result by choosing $f(x) = e^x$.
 \end{proof}

\begin{lemma}\label{lem:UI}
Consider a family $\{Y_n, n\in \mathbb N\}$ of non-negative and integrable random variables such that $\sup_{n}\E|Y_n-Y_n^\e|<\e$, where  $\{Y_n^\e, n\in \mathbb N, \e \in(0,1)\}$  satisfies $\sup_{n}\E[|Y_n^\e|^2]=M_\e<\infty$. Then, $\{Y_n, n\in \mathbb N\}$ is uniformly integrable.
\end{lemma}
\begin{proof}
By the triangular inequality and the  Cauchy-Schwarz inequality, we get
\vspace{-0.1cm}
\begin{align*}
    \E |Y_n\mathbf 1_{[Y_n>k]}|&\le \E |Y_n\mathbf 1_{[Y_n>k]}-Y^\e_n\mathbf 1_{[Y_n>k]} |+\E |Y^\e_n\mathbf 1_{[Y_n>k]}-Y^\e_n\mathbf 1_{[Y^\e_n>k]} |+\E |Y^\e_n\mathbf 1_{[Y^\e_n>k]}|\\
    &\le \E|Y_n-Y_n^\e|+ \big(\E|Y_n^\e|^2\big)^{\frac12}\big[\bP(Y_n>k)^{\frac12}+\bP(Y_n^\e>k)^{\frac12}\big]+\E |Y^\e_n\mathbf 1_{[Y^\e_n>k]}|\\
    &\le \e + \sqrt{M_\e}\big[\bP(Y_n>k)^{\frac12}+\bP(Y_n^\e>k)^{\frac12}\big] +\E |Y^\e_n\mathbf 1_{[Y^\e_n>k]}|.
\end{align*}

\vspace{-0.15cm}
\noindent For any fixed $\e\in(0,1)$, we may find $K=K(\e)$ such that for all $k>K$ and $n\in\bN$, $\bP(Y_n>k)^{1/2}+\bP(Y_n^\e>k)^{1/2}<\e(M_\e)^{-1/2}$ and $\E |Y^\e_n\mathbf 1_{[Y^\e_n>k]}|<\e$ by Markov's inequality and the Cauchy-Schwarz inequality. Hence, $\sup_{k\geq K, {n}\in\bN}\E |Y_n\mathbf 1_{[Y_n>k]}|<3\e$. This proves the uniform integrability of $\{Y_n, n\in\bN\}$.
\end{proof}
\end{appendix}

\vspace{-0.5cm}
\section*{Acknowledgments}
We thank Quentin Berger, Rongfeng Sun, and Jinjiong Yu for valuable suggestions on an earlier version of the paper. We also thank Julien Poisat for helpful comments and references. We are further grateful to the referees for their suggestions and  comments, which have significantly improved the paper.

\vspace{-0.2cm}
\section*{funding}
J. Song is partially supported by NSFC (No. 12471142) and the Fundamental Research Funds for the Central Universities. R. Wei is supported by NSFC (No. 12401170) and Xi'an Jiaotong-Liverpool University Research Development Fund (No. RDF-23-01-024).

\bibliographystyle{plain}
 \bibliography{ref}

\end{document}